\documentclass[onefignum,onetabnum]{siamart251216}

\usepackage{lipsum}
\usepackage{amsfonts}
\usepackage{graphicx}
\usepackage{epstopdf}
\usepackage{algorithmic}
\usepackage{enumitem}
\usepackage{dsfont}

\ifpdf
  \DeclareGraphicsExtensions{.eps,.pdf,.png,.jpg}
\else
  \DeclareGraphicsExtensions{.eps}
\fi

\newsiamremark{remark}{Remark}
\newsiamremark{hypothesis}{Hypothesis}
\crefname{hypothesis}{Hypothesis}{Hypotheses}
\newsiamthm{claim}{Claim}
\newsiamthm{assumption}{Assumption}
\newsiamremark{fact}{Fact}
\crefname{fact}{Fact}{Facts}

\newcommand{\dist}{\mathrm{dist}}
\newcommand{\sgn}{\mathrm{sgn}}

\makeatletter
\newcommand{\oset}[3][0ex]{%
  \mathrel{\mathop{#3}\limits^{
    \vbox to#1{\kern-2\ex@
    \hbox{$\scriptstyle#2$}\vss}}}}
\makeatother

\newcommand{\weakly}{\rightharpoonup}
\newcommand{\weaklystar}{\stackrel\star\rightharpoonup}

\newcommand{\supp}{\mathrm{supp}}
\newcommand{\diam}{\mathrm{diam}}

\newcommand{\dd}{\mathrm{d}}

\newcommand{\N}{\mathbb{N}}

\newcommand{\R}{\mathbb{R}}
\newcommand{\Z}{\mathbb{Z}}

\newcommand{\DD}{\mathcal{D}}

\newcommand{\II}{\mathcal{I}}

\newcommand{\KK}{\mathcal{K}}
\newcommand{\LL}{\mathcal{L}}
\newcommand{\MM}{\mathcal{M}}

\newcommand{\ZZ}{\mathcal{Z}}

\usepackage{xcolor}
\definecolor{darkgreen}{rgb}{0,0.5,0}

\headers{Differentiability for 
bilateral parabolic EVIs}
{S. Blomeyer and C. Christof}

\title{Directional differentiability for the solution map of the bilateral parabolic obstacle problem\thanks{Submitted to the editors DATE.
}}

\author{%
Simon Blomeyer\thanks{Technische Universit\"{a}t Darmstadt, Department of Mathematics, Dolivostraße 15, 64293 Darmstadt, Germany,  
\email{blomeyer@mathematik.tu-darmstadt.de}, \email{christof@mathematik.tu-darmstadt.de}.}
\and  
Constantin Christof\footnotemark[2]
}

\usepackage{amsopn}

\begin{document}

\maketitle

\begin{abstract}
We prove that the solution map of the 
classical bilateral parabolic obstacle problem 
is directionally differentiable when interpreted 
as a function between suitable Lebesgue 
and Sobolev spaces. 
To the best of our knowledge, this paper is 
the first to establish this kind of differentiability 
for a parabolic evolution
variational inequality with two-sided fully 
distributed pointwise inequality constraints. 
Our analysis relies on 
known differentiability results for 
unilateral parabolic
obstacle problems
and a new
theorem on the directional differentiability 
of solution operators of 
non-expansive parameterized 
fixed-point equations. The latter may also be 
of interest
for other applications. 
Our results can be used, for example, 
to derive optimality conditions for 
optimal control problems governed by bilateral 
parabolic obstacle problems. 
\end{abstract}

\begin{keywords}
directional differentiability, 
parabolic obstacle problem, 
bilateral constraint, 
fixed-point equation, 
implicit function,
optimality conditions, 
stationarity system,
optimal control
\end{keywords}

\begin{MSCcodes}
35K85, 35R35, 49J40, 49K40, 58E35 
\end{MSCcodes}

\section{Introduction}
This paper is concerned with the 
differentiability properties of the 
solution map $S_\psi^\phi\colon u \mapsto y$
of the classical bilateral parabolic obstacle problem
\begin{equation}
\label{eq:bilateral_obst_intro}
\left \{
~
\begin{aligned}
&\text{Find } y \in \KK_\psi^\phi \cap H^1(H^{-1})
\text{ such that } y(0) = y_0 \text{ and }
\\
&\int_0^T \left \langle \partial_t y  - \Delta y - u, v - y\right  \rangle_{H_0^1(\Omega)}\mathrm{d}t \geq 0 
~\quad\forall v \in \KK_\psi^\phi.
\end{aligned}
\right.
\end{equation}
Here, $\Omega \subset \R^d$,
$d \in \{1,2,3\}$, is a bounded domain that is convex or of class $C^{1,1}$, 
$T>0$ is a fixed final time,
$y_0$ 
is a given 
initial value, 
and the admissible set $\KK_\psi^\phi$
is defined by 
$\smash{\KK_\psi^\phi := \{ v \in L^2(H_0^1) \mid 
\psi \leq v \leq \phi \text{ a.e.\ in } (0,T) \times \Omega\}}$
with two sufficiently smooth 
obstacle functions $\psi$ and $\phi$. 
(See Assumption \ref{ass:standing:obstacle} for the precise setting
and \cref{sec:2} for the definitions of the involved function spaces, dual pairings, derivatives, etc.) 
The main result of this work---\cref{th:main}---establishes that the solution map 
$S_\psi^\phi\colon u \mapsto y$ of \eqref{eq:bilateral_obst_intro} is directionally 
differentiable when considered as a function between suitable 
Lebesgue and Sobolev spaces.
It relies on a new theorem on the directional 
differentiability of solution operators of parameterized 
non-expansive fixed-point equations 
in duals of separable normed spaces 
that we derive in an abstract setting and 
that is also interesting on its own; see \cref{th:abstract_dir_diff}. The differentiability results 
that we obtain  for \eqref{eq:bilateral_obst_intro} 
open up the possibility 
to derive 
optimality conditions of Bouligand- and strong-stationarity-type  for optimal control problems governed by 
evolution variational inequalities (EVIs) 
of the form
\eqref{eq:bilateral_obst_intro};
cf.\ \cite[sect.~5]{Christof2019}. 

Before we begin with our analysis, let us 
give some background: While directional differentiability 
results for the elliptic obstacle problem 
have been known since the nineteen-seventies \cite{Haraux1977,Mignot1976}, 
it was an open question 
for many years whether 
the solution map of the parabolic obstacle problem
is directionally differentiable or not. 
See, e.g., 
the comments on this topic in 
\cite[Remarque 3]{MignotPuel1984:2},
\cite[sect.~4]{MignotPuel1984}, and \cite[sect.~7.5]{BonnansShapiro2000}.
The main difficulty in deriving 
differentiability results for 
solution maps of parabolic EVIs of the type
\eqref{eq:bilateral_obst_intro} is that it 
is not possible 
to control the time derivatives of the 
difference quotients of the solution map
$\smash{S_\psi^\phi\colon u \mapsto y}$ in 
suitable function spaces. Because of this effect, 
accumulation points of difference quotients 
of $S_\psi^\phi$ can be functions that are discontinuous
with respect to (w.r.t.) time---in contrast to the solutions of 
\eqref{eq:bilateral_obst_intro} which are 
continuous---and the machinery 
that was developed for the elliptic case in 
\cite{Haraux1977,Mignot1976}
becomes inapplicable; see the 
counterexamples in \cite[sect.~3]{Christof2019}
and \cite[Example 4.1]{Brokate2021}. 
(We remark that, if the 
constraints in $\KK_\psi^\phi$ are only imposed on the
boundary $\partial \Omega$, then it is possible
to transfer elements of the elliptic approach to the parabolic setting;
see \cite{Jarusek2003}.) 
Only quite recently, it has been 
demonstrated in \cite{Christof2019}
that directional differentiability results
for parabolic obstacle-type EVIs can be 
derived along different lines, namely,
by exploiting pointwise-a.e.\ 
convexity and concavity properties;
see \cref{lem:dir_diff_uni} below and 
also \cite{Brokate2015} where a similar idea has been used for scalar hysteresis operators. 
The main drawback of this approach is that 
it works only for problems 
with unilateral constraints,
i.e., EVIs \eqref{eq:bilateral_obst_intro}
with $\psi \equiv - \infty$ or $\phi \equiv \infty$.
The derivation of directional differentiability results 
for the parabolic case with two-sided 
pointwise constraints in $(0,T) \times \Omega$ 
has thus remained an open problem.

In the present paper, 
we show how such results can be obtained. 
The main idea of our analysis
is to identify the multiplier map 
of the bilateral parabolic obstacle problem \eqref{eq:bilateral_obst_intro}
(see \cref{def:sol_ops}) with an implicit 
function that solves a system of equations 
which involves the multiplier maps of the 
unilateral EVIs associated with the 
obstacles $\psi$ and $\phi$; see \cref{lem:random_lem25}. 
By exploiting that the 
latter multiplier maps are 
directionally differentiable, 
satisfy a special Lipschitz estimate, 
and can be extended to 
spaces of Radon measures, we are then able 
to obtain directional differentiability 
for the bilateral case by means of a general 
result on the differentiability properties of 
solution operators of parameterized 
fixed-point equations; 
see \cref{th:abstract_dir_diff} and \cref{prop:dir_diff}. 
We remark that, in the elliptic setting, 
a similar idea has recently been used 
in \cite{ChristofWachsmuth2025} to show 
that $H_0^1(\Omega)$-projections 
onto sets with pointwise bounds are 
Newton differentiable. 
The analysis for the parabolic case, however, 
differs quite significantly from that in 
 \cite{ChristofWachsmuth2025}
 as one cannot work in a Hilbert 
 space setting when studying  \eqref{eq:bilateral_obst_intro}
 and, as a consequence, cannot invoke the  
 results on angled subspaces 
 derived in  \cite[sect.~5]{ChristofWachsmuth2025}.

 What is quite remarkable is that 
 the argumentation that we use to prove the directional
 differentiability of the solution map
 $S_\psi^\phi\colon u \mapsto y$ of \eqref{eq:bilateral_obst_intro} 
actually exploits the presence of the 
 time derivative in \eqref{eq:bilateral_obst_intro}
 and, thus, relies 
 on precisely those terms that 
 are the most problematic when trying to argue along the classical lines of 
 \cite{Haraux1977,Mignot1976}.
 The main observation is that, 
 when deriving Lipschitz 
estimates for multiplier maps by 
 ``testing'' \eqref{eq:bilateral_obst_intro} 
 with mollified signum functions, 
 the time derivative produces extra terms 
 that result in stronger non-expansiveness properties;
 see \cref{lem:mult_contractive}. 
 These properties are 
 the key for the application of 
 our general results 
 on the 
 directional differentiability 
 of solution maps of
 parameterized fixed-point equations. 
 We remark that, 
 along similar lines, 
 it might also be possible to 
 tackle the sensitivity analysis 
 of other parabolic nonsmooth systems.
 We leave this topic  for future research. 
 
 As already mentioned,
 the main application area for the 
 differentiability results 
 derived in this paper is the field of 
 optimal control of parabolic EVIs.
 We expect that, 
 with \cref{th:main} at hand, 
 it is quite straightforward to 
 derive 
 Bouligand-type 
 optimality conditions
 and 
 stationarity systems 
 involving complementarity conditions 
 similar to those in
 \cite{MignotPuel1984}
 for 
 optimization problems 
 constrained by \eqref{eq:bilateral_obst_intro};
 see the analysis for the unilateral
 case in  
 \cite[sect.~5]{Christof2019}.
To maintain the focus of this paper, a detailed discussion of this topic is deferred to future work.
 We conclude this introduction with a short
 overview of the structure of the remainder of the paper:
 
 In \cref{sec:2}, we discuss basic concepts and introduce the used notation. 
 \Cref{sec:3} is concerned with 
 the derivation of the abstract 
 directional differentiability 
 result for 
 non-expansive fixed-point equations 
 that our analysis of \eqref{eq:bilateral_obst_intro}
 is based on. See 
 \cref{th:abstract_dir_diff}
 for the main result of this section.
 \Cref{sec:4} collects preliminaries 
 on the parabolic obstacle problem 
 \eqref{eq:bilateral_obst_intro} that are 
 required for our approach. 
 In \cref{sec:5}, we establish 
 that \eqref{eq:bilateral_obst_intro}
 is indeed covered by the abstract theory
 of \cref{sec:3} and, by 
 invoking \cref{th:abstract_dir_diff},
 arrive at the main result of this 
 paper---\cref{th:main}.

\section{Notation}
\label{sec:2}
In this paper, we use the standard symbols 
$\N$, $\N_0$, $\Z$, and $\R$ for the natural numbers, 
the nonnegative integers, the
integers, and the real numbers, respectively. 
Norms and inner products on real vector spaces
are denoted by 
$\|\cdot\|$ and $(\cdot, \cdot)$,
respectively,
with a subscript that clarifies which space we refer to. 
If $X$ and $Y$ are normed spaces, then
$\LL(X,Y)$ is the space of linear and 
continuous functions from $X$ to $Y$, endowed 
with its canonical norm. 
In the special case $Y= \R$, 
we write 
$X^* := \LL(X,\R)$ for the topological dual space of $X$
and define $\langle x^*, x\rangle_X := x^*(x)$
for all $x^* \in X^*$, $x \in X$. 
For the modes of strong, weak, and weak-star 
convergence, we use the arrows
$\to$, $\weakly$, and $\weaklystar$, respectively.
The closed ball of radius $r > 0$ 
in a normed space $X$ with center $x \in X$ is denoted
by $B_r^X(x)$. 
If $X = \R^d$ for some $d \in \N$ and 
$\|\cdot\|_X$ is the Euclidean norm $|\cdot|$,
then we simply write $B_r(x)$. 
For arbitrary subsets $V_1, V_2$ of a normed space $X$, 
we define $\diam(V_1)$ to be the diameter of $V_1$,
$\partial V_1$ to be the boundary of $V_1$,
$\bar V_1$ to be the closure of $V_1$,
$\mathds{1}_{V_1} \colon X \to \{0,1\}$ to be the 
indicator function of $V_1$, 
and $\dist(V_1, V_2)$ to be the distance between
$V_1$ and $V_2$. 
If $F\colon X \to Y$ is a function between 
normed spaces $X$ and $Y$, then
$F^{-1}(W)$ denotes the preimage of a set $W \subset Y$,
$F|_V$ the restriction of $F$ to a set $V \subset X$,
and $F'(x;h) \in Y$ the directional derivative 
of $F$ at $x \in X$ in direction $h \in X$ (if it exists). 

For the real Lebesgue spaces on an arbitrary nonempty
measurable set $\Omega \subset \R^d$, $d \in \N$,
we use the standard symbols 
$L^p(\Omega)$, $1 \leq p \leq \infty$, in this paper.
If 
$\Omega$ is equal to the intersection of 
a closed and an open subset of $\R^d$ and 
equipped with the Euclidean metric,
then we further write 
$C(\Omega)$ for the space of 
continuous functions $v\colon \Omega \to \R$, 
$C_c(\Omega)$ for the 
space of all $v \in C(\Omega)$ 
with compact support $\supp(v)$ 
in $\Omega$, 
$C_0(\Omega)$ for the closure of $C_c(\Omega)$
w.r.t.\ the supremum norm $\|\cdot\|_\infty$,
and $C_c^\infty(\Omega)$ for the set 
$\{v \in C_c(\Omega) \mid v = \hat v|_\Omega \text{ for some }
\hat v \in C^\infty(\R^d)\}$, where 
$C^\infty(\R^d)$ denotes the space of smooth 
real-valued functions on $\R^d$. 
Recall that, for sets $\Omega$ of the latter type, 
$(C_0(\Omega), \|\cdot\|_\infty)$ is a Banach space, 
the Riesz--Alexandroff representation theorem
\cite[Theorem 2.4.6]{Attouch2006} implies that 
$C_0(\Omega)^*$ is isometrically isomorphic 
to the space $\MM(\Omega)$ 
of finite signed Radon measures
with 
$\|\mu\|_{C_0(\Omega)^*} = \|\mu\|_{\MM(\Omega)} = |\mu|(\Omega)$
for all $\mu \in \MM(\Omega)$,
and we have  $L^1(\Omega) \hookrightarrow \MM(\Omega)$
 with $\|v\|_{L^1(\Omega)} = \|v\|_{\MM(\Omega)}$ for all 
$v \in L^1(\Omega)$. Here, the 
arrow $\hookrightarrow$ indicates a continuous embedding
and  $|\mu|$ 
denotes the total variation measure of $\mu \in \MM(\Omega)$. 
If $v \in C(\Omega)$ is given, then we
also write 
$\{v = 0\}$ for the level set
$v^{-1}(\{0\})$. 
Other preimages are denoted analogously, 
e.g., $\{v \neq 0\} := v^{-1}(\R \setminus \{0\})$. 
If $v \in L^p(\Omega)$ for some $1 \leq p \leq \infty$,
then we use the same shorthand notation for preimages 
with the convention 
that the resulting sets are understood as being 
defined up to sets of Lebesgue measure zero. 

Suppose now that $\emptyset \neq \Omega \subset \R^d$, $d \in \N$, 
is open and bounded and that \mbox{$T>0$}. 
In what follows, 
we use the standard symbols 
$C^k(\Omega)$, $C^{k,\alpha}(\Omega)$,
$W^{k,p}(\Omega)$, $H^k(\Omega)$, $H_0^k(\Omega)$,
$L^p(0,T;W^{k,q}(\Omega))$,
$W^{k,p}(0,T;L^q(\Omega))$, $C([0,T];L^p(\Omega))$, etc.\ with 
$k \in \N$, $\alpha \in (0,1]$, and $p,q \in [1,\infty]$
for the spaces of classically differentiable functions,
the Hölder, Sobolev, and Sobolev--Bochner spaces, 
and the space of continuous functions
on $[0,T]$ with values in $L^p(\Omega)$,
respectively. For the sake of brevity, 
we drop the reference to the interval $(0,T)$ and 
the set $\Omega$
in the notation for the spaces involving time 
if it is clear, i.e., we write 
$L^2(L^2)$ instead of 
$L^2(0,T;L^2(\Omega))$, $C(L^2)$ instead of 
$C([0,T];L^2(\Omega))$, etc. 
For the dual of $H_0^1(\Omega)$ with pivot space $L^2(\Omega)$,
we write $H^{-1}(\Omega)$ as usual. 
Recall that the space 
$L^p(L^p)$ can be identified (in the sense of an isometric isomorphism)
with the space $L^p(  (0,T) \times \Omega )$ 
for all 
$p \in [1,\infty)$. 
For time derivatives, we use the symbol 
$\partial_t$ (understood distributionally or, respectively, 
in the sense of the Sobolev--Bochner spaces),
and for gradients and the Laplacian 
(w.r.t.\ the spatial variable in the Sobolev--Bochner setting),
we use the standard notation $\nabla$ and $\Delta$. 
In the one-dimensional case, (weak) derivatives are also 
denoted by $v'$ instead of $\nabla v$. 
Evaluations of elements of the spaces 
$W^{1,p}(L^q)$ at times $t \in [0,T]$ are always 
understood in the sense of traces. The same is true 
for evaluations of elements of $W^{k,p}(\Omega)$
on $\partial \Omega$ if $\Omega$ is a
Lipschitz domain. With $\sgn\colon \R \to \{-1,0,1\}$,
we denote the usual signum function. Note that, 
when writing $f(v)$ for some $f \colon \R \to \R$
and an element of a Lebesgue/Bochner/Sobolev space $v$,
this always means that $f$ acts by superposition. 

\section{An abstract directional differentiability result}
\label{sec:3}

In this section, 
we derive the 
abstract directional differentiability result for 
parameterized non-expansive fixed-point equations
that allows us to prove the 
directional differentiability of the 
solution map $S_\psi^\phi\colon u \mapsto y$ 
of the bilateral parabolic obstacle problem 
\eqref{eq:bilateral_obst_intro}. Henceforth, we focus on 
equations of the following type:
\begin{equation*}
    \tag{F}
    \label{eq:F}
    \qquad \lambda = F(\lambda,u).
\end{equation*}
Our standing assumptions on the quantities in \eqref{eq:F} are as follows:
\begin{assumption}[standing assumptions for \cref{sec:3}]~
\label{ass:standing:abstract}~
\begin{enumerate}[label=\roman*)]
\item $(\Lambda, \|\cdot\|_\Lambda)$ is the 
topological dual space of a real separable 
normed space;
\item $(Y, \|\cdot\|_Y)$ and 
$(U, \|\cdot\|_U)$ are real normed spaces;
\item $\bar u \in U$ and $r>0$ are given;
\item\label{ass:standing:abstract:iv}
$F\colon \Lambda \times B_r^U(\bar u) \to \Lambda$ 
is a function such that the following 
conditions hold:
\begin{enumerate}[label=\alph*)]
\item there 
exists a map
$M\colon B_r^U(\bar u) \to \Lambda$ such that 
\begin{equation}
\label{eq.M_FP}
    M(u) = F(M(u), u) \qquad \forall u \in B_r^U(\bar u)
\end{equation}
and such that there exists a constant $C>0$ with 
\begin{equation}
\label{eq:LipM}
\| M(u) - M(\bar u)\|_\Lambda 
\leq C \|u - \bar u\|_U\qquad \forall u \in B_r^U(\bar u);
\end{equation}
\item there exists 
a function
$G\colon \Lambda \times B_r^U(\bar u) \to Y$ 
such that,
for all $\lambda_1, \lambda_2 \in \Lambda$ 
and $u \in B_r^U(\bar u)$, we have
\begin{equation}
\label{eq:Lipschitz_estimate_abstract}
 \|
F(\lambda_1,u) - F(\lambda_2,u)
 \|_\Lambda
+
 \|
G(\lambda_1,u) - G(\lambda_2,u)
\|_Y
\leq
 \|
\lambda_1 - \lambda_2
 \|_\Lambda;
\end{equation}
\item $F$ and $G$ are Hadamard directionally differentiable 
at $(M(\bar u), \bar u)$
in the sense that, for all 
$g \in \Lambda$ and $h \in U$, 
there exist $F'[(M(\bar u),\bar u);(g,h)] \in  \Lambda$
and $G'[(M(\bar u),\bar u);(g,h)] \in  Y$
with
\begin{equation*}
\begin{aligned}
    &\{g_n\} \subset \Lambda, 
    \{\tau_n\} \subset \left (0, \frac{r}{1+\|h\|_U} \right ),
    g_n \weaklystar g \text{ in } \Lambda, 
    \tau_n \to 0
    \\
    &
    \Rightarrow
    ~
    \left \{
    \begin{aligned}
    \frac{F(M(\bar u) + \tau_n g_n, \bar u + \tau_n h) - F(M(\bar u),\bar u)}{\tau_n}
    \weaklystar F'[(M(\bar u),\bar u);(g,h)],
    \\
    \frac{G(M(\bar u) + \tau_n g_n, \bar u + \tau_n h) - G(M(\bar u),\bar u)}{\tau_n}
    \to G'[(M(\bar u),\bar u);(g,h)].
    \end{aligned}
    \right.
\end{aligned}
\end{equation*}
\end{enumerate}
\end{enumerate}
\end{assumption}

We begin our analysis by proving that 
the Lipschitz estimate \eqref{eq:Lipschitz_estimate_abstract} 
carries over to the directional derivatives of $F$ and $G$.

\begin{lemma}%
\label{lem:pseudo_contractive}%
For all $h \in U$ and $g_1, g_2 \in \Lambda$, 
it holds 
\begin{equation}
\label{eq:randomeq273g38}
\begin{aligned}
& \|
F'[(M(\bar u),\bar u);(g_1,h)] 
- F'[(M(\bar u), \bar u);(g_2,h)]
 \|_\Lambda
\\
&\quad +
 \|
G'[(M(\bar u),\bar u);(g_1,h)] 
- G'[(M(\bar u),\bar u);(g_2,h)]
  \|_Y
\leq
 \|
g_1 - g_2
 \|_\Lambda.
\end{aligned}
\end{equation}
\end{lemma}
\begin{proof}
If $h \in U$ and $g_1, g_2 \in \Lambda$ are given, 
then  
Assumption
\ref{ass:standing:abstract}\ref{ass:standing:abstract:iv}b)
implies that, for all sufficiently small $\tau>0$, we have
\begin{align*}
&\left \|
\frac{
F(M(\bar u) +  \tau g_1, \bar u + \tau h)  - F(M(\bar u),\bar u) + F(M(\bar u),\bar u) - F(M(\bar u) + \tau g_2,\bar u + \tau h)}
{\tau}
\right \|_\Lambda
\\
&
+ \hspace{-0.01cm}
\left \|
\frac{
G(M(\bar u) +  \tau g_1, \bar u + \tau h) 
- G(M(\bar u), \bar u) 
+ G(M(\bar u), \bar u)
- G(M(\bar u) +  \tau g_2, \bar u + \tau h)}
{\tau}
\right \|_Y
\\
&
\leq
\|
g_1 - g_2
 \|_\Lambda.
\end{align*}
Letting $\tau$ go to zero in the above---using the 
directional differentiability properties 
of $G$ and $F$ 
and the weak-star lower semicontinuity of 
$\|\cdot\|_\Lambda$---yields \eqref{eq:randomeq273g38}. 
\end{proof}

To study the directional
differentiability properties of 
$M$ at $\bar u$, we next introduce:

\begin{definition}[difference quotients of $M$]%
\label{def:diff_quots}%
If $h \in U$
is given, then 
we define $\hat \tau_h := r/(1+\|h\|_U) \in (0, \infty)$ and 
denote by $\{\delta_\tau^h\} \subset \Lambda$
the family of difference quotients
 \[
 \delta_\tau^h := \frac{M(\bar u + \tau h) - M(\bar u)}{\tau},
 \qquad 0 < \tau < \hat \tau_h.
 \]
\end{definition}

Note that
$\{\delta_\tau^h\}$ is bounded in $\Lambda$
for all $h \in U$ by \eqref{eq:LipM}. 
This allows us to prove:

\begin{lemma}
\label{lem:DD_FP}
Let $h \in U$ be given.
The set 
\begin{equation}
\label{eq:DD_def}
\DD_h := 
\left \{ \delta \in \Lambda ~\left |~ \exists \{\tau_n\} \subset (0, \hat\tau_h) \colon \tau_n \to 0, 
\delta_{\tau_n}^h \weaklystar \delta \text{ in } \Lambda \right. \right\} 
\end{equation}
is nonempty and bounded 
and every $\delta \in \DD_h$ satisfies
\begin{equation}
\label{eq:FP}
\delta = F'[(M(\bar u), \bar u);(\delta, h)].
\end{equation}
\end{lemma}
\begin{proof}
    The nonemptiness and boundedness of $\DD_h$ follow
    from the theorem of Banach--Alaoglu
    and the boundedness of 
    the family $\{\delta_\tau^h\}$ in $\Lambda$ obtained from \eqref{eq:LipM}.
    Consider now some arbitrary but fixed
    $\delta \in \DD_h$ with associated sequence 
    $\{\tau_n\} \subset (0, \hat \tau_h)$.
    Then it follows from
    \cref{def:diff_quots} and \eqref{eq.M_FP} that
    \begin{equation*}
    \begin{aligned}
    \delta_{\tau_n}^h = 
    \frac{M(\bar u+ \tau_n h)-M(\bar u)}{\tau_n}
    &=\frac{F(M(\bar u+\tau_nh), \bar u+\tau_nh)- F(M(\bar u), \bar u)}{\tau_n}\\
    &= \frac{F(M(\bar u)+\tau_n \delta_{\tau_n}^h, \bar u+\tau_nh)-F(M(\bar u), \bar u)}{\tau_n}.
    \end{aligned}
    \end{equation*}
    Due to Assumption 
\ref{ass:standing:abstract}\ref{ass:standing:abstract:iv}c),
    we can pass to the limit $n \to \infty$ in the above to obtain
    $ \delta = F'[(M(\bar u), \bar u);(\delta, h)] $.
    As $\delta \in \DD_h$ was arbitrary, this proves the assertion.
\end{proof}

We remark that 
one cannot expect that the fixed-point equation 
\eqref{eq:FP} is
uniquely solvable in the situation of 
Assumption \ref{ass:standing:abstract}. 
(Consider, e.g., the case 
$Y = U = \Lambda = \R$, 
$G\equiv 0$,
and $F(\lambda, u) = \lambda$.) However, 
due to \cref{lem:pseudo_contractive}, we still get:

\begin{lemma}
\label{lem:same_for_accpts}
Let $h \in U$ be given and let $\DD_h$ be defined as in \eqref{eq:DD_def}. 
Then, 
for all $\delta_1, \delta_2 \in \DD_h$, it 
holds 
$
G'[(M(\bar u), \bar u);(\delta_1, h)] 
= 
G'[(M(\bar u), \bar u);(\delta_2, h)]
$.
\end{lemma}
\pagebreak
\begin{proof}
Let $\delta_1, \delta_2 \in \DD_h$ be fixed. 
Then \eqref{eq:FP} and \cref{lem:pseudo_contractive} imply
\begin{align*}
& \|
\delta_1 - \delta_2
 \|_\Lambda
 +
 \|
G'[(M(\bar u),\bar u);(\delta_1,h)] 
- 
G'[(M(\bar u),\bar u);(\delta_2,h)]
\|_Y
\\
&=
 \|
F'[(M(\bar u),\bar u);(\delta_1,h)] 
- 
F'[(M(\bar u),\bar u);(\delta_2,h)]
\|_\Lambda
\\
&\qquad +
\|
G'[(M(\bar u),\bar u);(\delta_1,h)] 
- 
G'[(M(\bar u),\bar u);(\delta_2,h)]
\|_Y
\\
&\leq
\|
\delta_1 - \delta_2
 \|_\Lambda.
\end{align*}
Rearranging the above yields
$
G'[(M(\bar u), \bar u);(\delta_1, h)] 
= 
G'[(M(\bar u), \bar u);(\delta_2, h)]
$. 
\end{proof}

With \cref{lem:same_for_accpts} at hand, 
we can now prove the main result of this section.

\begin{theorem}[directional differentiability for the fixed-point equation]\label{th:abstract_dir_diff}%
Define $S\colon B_r^U(\bar u) \to Y$, 
$S(u):=G(M(u), u)$. Then $S$ 
is directionally differentiable 
at $\bar u$
in all directions $h \in U$
in the sense that, for all $h \in U$, 
there exists $S'(\bar u;h) \in Y$
such that 
\begin{equation}
\label{eq:diff_claim}
\frac{S(\bar u + \tau h) - S(\bar u)}{\tau}
\to S'(\bar u;h)
\text{ in } Y
\text{ for } (0, \hat \tau_h) \ni \tau \to 0.
\end{equation}
Further, 
for every $h \in U$, the directional derivative
$S'(\bar u;h)$ is given by 
\[
   S'(\bar u;h) = 
   G'[(M(\bar u), \bar u);(\delta, h)]\qquad \forall \delta \in \DD_h, 
\]
where $\DD_h$ denotes the set of weak-star accumulation points in 
\eqref{eq:DD_def}. 
\end{theorem}
\begin{proof}
Let $h \in U$ be fixed and let 
$\{\tau_n\} \subset (0, \hat\tau_h)$
be a sequence satisfying 
$\tau_n \to 0$ for $n \to \infty$.
Let $\delta_n^h := \delta_{\tau_n}^h$ be defined 
as in \cref{def:diff_quots}.
Then 
 $\{\delta_{n}^h\}$ is 
bounded in $\Lambda$ 
by \eqref{eq:LipM} and 
we can find a
subsequence $\{\delta_{n_k}^h\}$
such that 
\smash{$\delta_{n_k}^h \weaklystar \delta$} holds  for $k \to \infty$ with 
some $\delta \in \DD_h$. 
From the 
definition of $\delta_{n_k}^h$
and Assumption
\ref{ass:standing:abstract}\ref{ass:standing:abstract:iv}c), we get that 
   \begin{equation}
   \label{eq:randomeq2783g38}
    \begin{aligned}
    &\frac{G(M(\bar u + \tau_{n_k} h), \bar u+\tau_{n_k} h)- G(M(\bar u), \bar u)}{\tau_{n_k}}
    \\
    &\quad = \frac{G(M(\bar u)+\tau_{n_k} \delta_{n_k}^h, \bar u + \tau_{n_k} h)-G(M(\bar u), \bar u)}{\tau_{n_k}}
    \to G'[(M(\bar u),\bar u);(\delta,h)]
    \text{ in } Y.
    \end{aligned}
    \end{equation}
Due to \cref{lem:same_for_accpts}, 
we know that the right-hand side of \eqref{eq:randomeq2783g38}
is independent of $\delta \in \DD_h$, i.e., 
it holds $\eta = G'[(M(\bar u),\bar u);(\delta,h)]$
for all $\delta \in \DD_h$ with some 
fixed $\eta \in Y$.
We may thus conclude that
there exists 
$\eta \in Y$ such that, for all 
$\{\tau_n\} \subset (0, \hat \tau_h)$
with $\tau_n \to 0$, there is a 
subsequence $\{\tau_{n_k}\}$ satisfying 
   \begin{equation*}
    \begin{aligned}
    \frac{S(\bar u + \tau_{n_k} h) - S(\bar u)}{\tau_{n_k}}
    =
    \frac{G(M(\bar u + \tau_{n_k} h), \bar u+\tau_{n_k} h)- G(M(\bar u), \bar u)}{\tau_{n_k}}
    \to \eta
    \text{ in } Y.
    \end{aligned}
    \end{equation*}
Using a standard subsequence-subsequence argument, 
it now follows that \eqref{eq:diff_claim} 
holds with the derivative $S'(\bar u;h) = \eta$. 
This proves the assertion. 
\end{proof}

Note that---although the whole
idea of the proof of 
\cref{th:abstract_dir_diff} is to study
the set of weak-star accumulation points 
$\DD_h$ in \eqref{eq:DD_def}---we 
do not get in the above 
that the difference quotients 
$\{\delta_\tau^h\}$ converge weakly-star in $\Lambda$ for $\tau \to 0$. 
The directional differentiability 
is only obtained for the composition 
$S(u) := G(M(u), u)$.

\section{Preliminaries on parabolic obstacle problems}
\label{sec:4}
In the remainder of this paper, 
the main goal is to show that  
\cref{th:abstract_dir_diff}
can be applied to 
the EVI \eqref{eq:bilateral_obst_intro}. 
We begin by collecting several preliminary
results on unilateral and bilateral
parabolic obstacle problems that are essential for seeing that \eqref{eq:bilateral_obst_intro}
is covered by the abstract setting 
of \cref{sec:3}. 
Throughout this section, 
we always assume that the following conditions
are satisfied.

\begin{assumption}[standing assumptions and notation for \cref{sec:4}]~
\label{ass:standing:obstacle:prelims}~
\begin{enumerate}[label=\roman*)]
\item $T>0$ is a given terminal time;
\item $\Omega \subset \R^d$, $d \in \{1,2,3\}$, 
is a bounded domain that is 
convex or of class $C^{1,1}$;
\item $Q := (0,T) \times \Omega$ denotes the space-time cylinder;
\item $y_0 \in H_0^1(\Omega)$ is a given and fixed initial datum;
\item $\psi, \phi \in L^2(H^2) \cap H^1(L^2)$ are functions such that
$\psi(0) \leq y_0 \leq \phi(0)$ holds 
a.e.\ in $\Omega$ and such that there 
exists 
$v \in L^2(H_0^1)$ satisfying $ \psi \leq v \leq \phi \text{ a.e.\ in } Q$.
\end{enumerate}
\end{assumption}
For convenience, we also introduce the abbreviations
\begin{gather*}
\KK_\psi 
:= \{v \in L^2(H_0^1) \mid \psi \leq v  \text{ a.e.\ in } Q\},
\qquad
\KK^\phi  := - \KK_{-\phi},
\qquad 
\KK_\psi^\phi 
:= \KK_\psi \cap \KK^\phi.
\end{gather*}

Note that 
our assumptions 
on $\psi$ and $\phi$
ensure that 
the above three sets are nonempty.
We start with a solvability result for the parabolic obstacle problem.

\begin{proposition}[unique solvability for parabolic obstacle problems]
\label{prop:ex_uniqu}%
Suppose that
$u \in L^2(L^2)$ is given
and that 
$\KK \in \{\KK_\psi, \KK^\phi, \KK_\psi^\phi\}$. Then the problem%
\begin{equation}
\label{eq:bilateral_obst_weak}
\left \{
~
\begin{aligned}
&\text{Find } y \in \KK \cap H^1(H^{-1})
\text{ such that } y(0) = y_0 \text{ and }
\\
&\int_0^T \left \langle \partial_t y  - \Delta y - u, v - y\right  \rangle_{H_0^1(\Omega)}\mathrm{d}t \geq 0 
~\quad\forall v \in \KK
\end{aligned}
\right.
\end{equation}
has a unique solution $y$. 
Further, the following is true: 
\begin{enumerate}[label=\roman*)]
\item 
A function $y$ solves \eqref{eq:bilateral_obst_weak} 
for $\KK = \KK_\psi$
if and only if 
there exists
a multiplier $\lambda_\psi$
such that the following complementarity system holds:
\begin{equation}
\label{eq:unilateral_lower_obst_stat_sys}
\left \{
~
\begin{gathered}
y \in L^2(H_0^1\cap H^2) \cap H^1(L^2),
\quad 
\lambda_\psi \in L^2(L^2),
\quad y(0) = y_0, 
\\
\partial_t y  - \Delta y  = 
u - \lambda_\psi \quad \text{ a.e.\ in }Q,
\\
\lambda_\psi \leq 0,
\qquad  
y - \psi \geq 0,
\qquad
\lambda_\psi(y - \psi) = 0\quad \text{ a.e.\ in }Q.
\end{gathered}
\right.
\end{equation}
\item
A function $y$ solves \eqref{eq:bilateral_obst_weak} 
for
$\KK = \KK^\phi$
if and only if 
there exists a multiplier 
$\lambda^\phi$
such that the following complementarity system holds:
\begin{equation}
\label{eq:unlateral_upper_obst_stat_sys}
\left \{
~
\begin{gathered}
y \in L^2(H_0^1\cap H^2) \cap H^1(L^2),
\quad 
\lambda^\phi \in L^2(L^2),
\quad y(0) = y_0, 
\\
\partial_t y  - \Delta y  = 
u - \lambda^\phi\quad \text{ a.e.\ in }Q,
\\
\lambda^\phi \geq 0,
\qquad  
y - \phi \leq 0,
\qquad
\lambda^\phi(y - \phi) = 0\quad \text{ a.e.\ in }Q.
\end{gathered}
\right.
\end{equation}
\item
A function $y$ solves \eqref{eq:bilateral_obst_weak} 
for
$\KK = \KK_\psi^\phi$
if and only if 
there exist multipliers $\lambda_\psi$ and $\lambda^\phi$
such that the following complementarity system holds:
\begin{equation}
\label{eq:bilateral_obst_stat_sys}
\left \{
~
\begin{gathered}
y \in L^2(H_0^1\cap H^2) \cap H^1(L^2),
\quad 
\lambda_\psi, \lambda^\phi \in L^2(L^2),
\quad y(0) = y_0, 
\\
\partial_t y  - \Delta y  = 
u - \lambda_\psi - \lambda^\phi\quad \text{ a.e.\ in }Q,
\\
\lambda_\psi \leq 0,
\qquad  
y - \psi \geq 0,
\qquad
\lambda_\psi(y - \psi) = 0\quad \text{ a.e.\ in }Q,
\\
\lambda^\phi \geq 0,
\qquad  
y - \phi \leq 0,
\qquad
\lambda^\phi(y - \phi) = 0\quad \text{ a.e.\ in }Q.
\end{gathered}
\right.
\end{equation}
\end{enumerate}
\end{proposition}

\begin{proof}
The proof 
is along standard lines; see, e.g., 
\cite[chapt.~2]{Hafemeyer2020}, \cite[sect.~2]{Peng2025}, 
and \cite[sect.~2]{Adams2002}. 
We include it for the convenience 
of the reader 
and since it does not seem to be present in 
the literature yet in a version that 
fits to the setting in 
Assumption \ref{ass:standing:obstacle:prelims} and 
\eqref{eq:bilateral_obst_weak}.
As the arguments are exactly the same 
for the three sets 
$ \KK_\psi$, $ \KK^\phi$,
and $\KK_\psi^\phi$, we focus on the case 
$\smash{\KK = \KK_\psi^\phi}$ in the following. 

That \eqref{eq:bilateral_obst_weak} 
with $\smash{\KK = \KK_\psi^\phi}$ can have at most one solution follows
from a standard argument: If 
$y_1, y_2$ both solve \eqref{eq:bilateral_obst_weak}
with $\KK = \KK_\psi^\phi$, then 
using $y_1$ as the test function 
in the EVI for $y_2$ and vice versa and adding the resulting 
inequalities yields 
\begin{equation}
\label{eq:uniqueness}
\begin{aligned}
0 &\geq 
\int_0^T \left \langle \partial_t (y_1 - y_2)  - \Delta (y_1 - y_2) , y_1 - y_2\right  \rangle_{H_0^1(\Omega)}\mathrm{d}t
\\
&= \frac{1}{2}\|(y_1 - y_2)(T)\|_{L^2(\Omega)}^2
+
\| \nabla (y_1 - y_2) \|_{L^2([L^2]^d)}^2.
\end{aligned}
\end{equation}
Due to $y_1, y_2 \in L^2(H_0^1)$ and Poincaré's inequality, 
\eqref{eq:uniqueness} implies $y_1 = y_2$. Thus, 
the solution of \eqref{eq:bilateral_obst_weak} 
with $\KK = \KK_\psi^\phi$
is indeed unique
if it exists. 

To prove the solvability of \eqref{eq:bilateral_obst_weak} 
and the equivalence of \eqref{eq:bilateral_obst_weak}
and \eqref{eq:bilateral_obst_stat_sys}
in the case $\KK = \KK_\psi^\phi$,
we consider the family of semilinear parabolic boundary value problems 
\begin{equation}
\label{eq:sem_approx}
\left \{~
\begin{aligned}
\partial_t y^\varepsilon  - \Delta y^\varepsilon  
+ \frac{1}{\varepsilon}\max(0, y^\varepsilon - \phi) 
+ \frac{1}{\varepsilon}\min(0, y^\varepsilon - \psi) 
&=  u &&\text{in }(0,T) \times \Omega,
\\
y^\varepsilon  &= y_0 &&\text{in }\{0\} \times \Omega,
\\
y^\varepsilon  &= 0 &&\text{on }  (0,T) \times \partial\Omega,
\end{aligned}\right.
\end{equation}
where $\max(0, \cdot)$, $\min(0, \cdot)$
act by superposition and $\varepsilon > 0$.
Using \cite[Theorem 19.16]{Arendt2026form},
the same arguments as in \eqref{eq:uniqueness}, 
and standard regularity results for the heat 
equation (as found in \cite[Proposition 2.3]{ChristofVexler2020}), 
one easily checks that 
\eqref{eq:sem_approx} possesses a unique solution 
$y^{\varepsilon} \in L^2(H_0^1\cap H^2) \cap H^1(L^2)$
for all $\varepsilon > 0$.
Note that, 
due to the regularity \mbox{$y^\varepsilon \in L^2(H_0^1)$}, 
\mbox{Stampacchia's} lemma
\cite[Theorem 5.8.2]{Attouch2006},
and our assumptions on $\psi$,
we have $\min(0, y^\varepsilon - \psi) \in L^2(H_0^1)$.
This allows us to use $\min(0, y^\varepsilon - \psi)$
as the test function in the weak formulation of 
\eqref{eq:sem_approx} to get
\begin{equation*}
\begin{aligned}
&\left (u,  \min(0, y^\varepsilon - \psi)\right )_{L^2(L^2)}
\\
&=
\int_0^T \langle 
\partial_t y^\varepsilon  - \Delta y^\varepsilon  
+ \frac{1}{\varepsilon}\max(0, y^\varepsilon - \phi) 
+ \frac{1}{\varepsilon}\min(0, y^\varepsilon - \psi) ,
\min(0, y^\varepsilon - \psi)
\rangle_{H_0^1(\Omega)} \dd t
\\
&= 
\int_0^T 
\left (
\partial_t (y^\varepsilon-\psi),
\min(0, y^\varepsilon - \psi)
\right )_{L^2(\Omega)}
+
\left (
\nabla (y^\varepsilon-\psi),
\nabla \min(0, y^\varepsilon - \psi)
\right )_{L^2(\Omega)^d} \dd t
\\
&\qquad + \frac{1}{\varepsilon}  \|\min(0, y^\varepsilon - \psi)  \|_{L^2(L^2)}^2 
+
\int_0^T 
\left (
\partial_t \psi  -\Delta \psi,
\min(0, y^\varepsilon - \psi)
\right )_{L^2(\Omega)}
\dd t
\\
&\geq
\frac{1}{\varepsilon} \|\min(0, y^\varepsilon - \psi) \|_{L^2(L^2)}^2 
-
  \| \partial_t \psi -\Delta \psi  \|_{L^2(L^2)}
  \| \min(0, y^\varepsilon - \psi) \|_{L^2(L^2)}.
\end{aligned}
\end{equation*}
Here, in the last estimate, we used the 
formulas for truncations in 
\cite[Lemma A.1]{ChristofVexler2020} and the Cauchy--Schwarz inequality.
Rearranging the above and exploiting the Cauchy--Schwarz inequality 
again
shows that the family of functions 
$
\{ \min(0, y^\varepsilon - \psi)/\varepsilon \}_{\varepsilon > 0}
$
is bounded in $L^2(L^2)$. Along the same lines, 
we also obtain the boundedness of the 
family 
$\{ \max(0, y^\varepsilon - \phi) /\varepsilon \}_{\varepsilon > 0}$
in $L^2(L^2)$. Due to the structure of \eqref{eq:sem_approx}
and the properties of the heat equation, 
this also provides us with the information 
that $\{y^\varepsilon\}_{\varepsilon > 0}$
is bounded in $L^2(H_0^1\cap H^2) \cap H^1(L^2)$. 
Consider now some sequence $\{\varepsilon_n\} \subset (0, \infty)$
with $\varepsilon_n \to 0$
such that $y_{n} := y^{\varepsilon_n} \weakly  y$
holds in $L^2(H_0^1\cap H^2) \cap H^1(L^2)$  
for some $ y$ and such that 
$\min(0, y_n - \psi)/\varepsilon_n \weakly  \lambda_\psi$
and 
$\max(0, y_n - \phi) /\varepsilon_n \weakly  \lambda^\phi$ holds in $L^2(L^2)$ for some $\lambda_\psi, \lambda^\phi$.
(Such a sequence $\{\varepsilon_n\}$ can be found
due to the theorem of Banach--Alaoglu.) Then it necessarily holds 
$\lambda_\psi \leq 0$ and $\lambda^\phi \geq 0$ a.e.\ in $Q$.
Further, we get from \eqref{eq:sem_approx} that
$\partial_t  y  - \Delta  y  = 
u - \lambda_\psi - \lambda^\phi $ holds
a.e.\ in $Q$ and that  $ y(0) = y_0$.
As the lemma of Aubin--Lions \cite[Lemma 7.7]{Roubicek2013}
implies 
$y_{n} \to y$ in $L^2(H_0^1)$,
the convergences 
$\min(0, y_n - \psi), \max(0, y_n - \phi) \to 0$
in $L^2(L^2)$ additionally provide us with $\psi \leq  y \leq \phi$
a.e.\ in $Q$. 
From
$\min(0, y_n - \psi)/\varepsilon_n \weakly  \lambda_\psi$
and 
$\max(0, y_n - \phi) /\varepsilon_n \weakly  \lambda^\phi$ in $L^2(L^2)$
and
$y_{n} \to y$ in $L^2(H_0^1)$,
we lastly also obtain that 
\begin{align*}
&0
\leq
\left (\frac1{\varepsilon_n}\min(0, y_n - \psi), y_n - \psi \right )_{L^2(L^2)}
\to
\left (\lambda_\psi,  y- \psi \right )_{L^2(L^2)},
\\
&0
\leq
\left (\frac1{\varepsilon_n}\max(0, y_n - \phi), y_n - \phi \right )_{L^2(L^2)}
\to
\left (\lambda^\phi,  y- \phi \right )_{L^2(L^2)}.
\end{align*}
In view of the already established inequalities 
$\lambda_\psi \leq 0$, 
$\lambda^\phi \geq 0$,
and
$\psi \leq  y \leq \phi$
a.e.\ in $Q$, the above implies 
$\lambda_\psi( y - \psi) = 0$
and 
$\lambda^\phi(y - \phi) = 0$ a.e.\ 
in $Q$. By collecting everything, we now 
see that $ y$, $\lambda_\psi$, and $\lambda^\phi$
satisfy all of the conditions in \eqref{eq:bilateral_obst_stat_sys}.
This shows that the system
\eqref{eq:bilateral_obst_stat_sys}
admits a solution triple 
$(y, \lambda_\psi, \lambda^\phi)$. 
To complete the proof for the case 
$\smash{\KK = \KK_\psi^\phi}$, it remains to show that,
if $(y, \lambda_\psi, \lambda^\phi)$
solves the system \eqref{eq:bilateral_obst_stat_sys}, then 
$ y$ is also a solution of 
\eqref{eq:bilateral_obst_weak} with 
$\KK = \KK_\psi^\phi$.
(The already proven uniqueness of solutions of \eqref{eq:bilateral_obst_weak}
then yields all of the assertions of the proposition.)
To this end, 
we first note that \eqref{eq:bilateral_obst_stat_sys} 
implies 
$y(0) = y_0$ and $y \in \KK_\psi^\phi \cap H^1(H^{-1})$.
If we multiply the 
PDE in \eqref{eq:bilateral_obst_stat_sys}
with $v -  y$ for 
an arbitrary
\smash{$v \in \KK_\psi^\phi$} and integrate,
then we further get
\begin{equation*}
\begin{aligned}
\int_0^T \left (\partial_t  y - \Delta  y - u, v - y \right )_{L^2(\Omega)} \dd t 
&= 
\int_0^T \int_\Omega -\lambda_\psi (v-y) - \lambda^\phi(v-y) \dd x\dd t
\\
&= \int_0^T \int_\Omega 
-\lambda_\psi (v-\psi) - \lambda^\phi(v-\phi)\dd x \dd t
\geq 0,
\end{aligned}
\end{equation*}
where, in the second line, 
we have exploited that 
$\lambda_\psi( y - \psi) = 0$
and 
$\lambda^\phi(y - \phi) = 0$ holds a.e.\ in $Q$
and that 
$v \in \KK_\psi^\phi$ 
and $\lambda_\psi \leq 0 \leq \lambda^\phi$ a.e.\ in $Q$.
If $(y, \lambda_\psi, \lambda^\phi)$ solves
\eqref{eq:bilateral_obst_stat_sys},
then $y$ thus indeed solves \eqref{eq:bilateral_obst_weak} 
with $\smash{\KK = \KK_\psi^\phi}$.
This completes the proof.
\end{proof}

The next lemma allows 
us to obtain more regularity
for the solution $y$ of \eqref{eq:bilateral_obst_weak}
if $u$, $\psi$, and $\phi$
are more than just $L^2(L^2)$-,
respectively, $L^2(H^2) \cap H^1(L^2)$-regular.

\begin{lemma}[parabolic Lewy--Stampacchia inequalities]%
\label{lem:par_LewyStamp}%
Let $u \in L^2(L^2)$ be given,
let $\KK \in \{\KK_\psi, \KK^\phi, \KK_\psi^\phi\}$,
and let $y$ denote the unique solution of 
\eqref{eq:bilateral_obst_weak}. 
Then:
\begin{enumerate}[label=\roman*)]
\item in the case $\KK = \KK_\psi$, it holds 
    \begin{equation}
    \label{eq:par_Lewy_Stamp_uni_lower}
       \begin{aligned}
            u
            \leq 
            \partial_t y
            -\Delta y
            \leq\max(\partial_t \psi -\Delta\psi, u) 
          \quad\text{ a.e.\ in }Q;
        \end{aligned}
        \end{equation}
\item in the case $\KK = \KK^\phi$, it holds 
    \begin{equation}
    \label{eq:par_Lewy_Stamp_uni_upper}
       \begin{aligned}
            \min(\partial_t \phi -\Delta\phi, u)
            \leq 
            \partial_t y
            -\Delta y
            \leq u
          \quad\text{ a.e.\ in }Q;
        \end{aligned}
        \end{equation}
\item in the case $\KK = \KK_\psi^\phi$, it holds 
    \begin{equation}
    \label{eq:par_Lewy_Stamp_bilateral}
       \begin{aligned}
            \min(\partial_t \phi -\Delta\phi, u)
            \leq 
            \partial_t y
            -\Delta y
            \leq\max(\partial_t \psi -\Delta\psi, u) 
          \quad\text{ a.e.\ in }Q.
        \end{aligned}
        \end{equation}
\end{enumerate}
\end{lemma}
\begin{proof}
To obtain the assertions of the lemma, 
we extend the arguments used in the proofs 
of \cite[Lemma~6.4]{ChristofWachsmuth2025} and \cite[Theorem~4.35]{Troianiello1987} to the parabolic setting. 

We begin with the case $\KK = \KK_\psi^\phi$:
Let $u \in L^2(L^2)$ be given and 
let $y$ be the solution of
\eqref{eq:bilateral_obst_weak}
with $\smash{\KK = \KK_\psi^\phi}$. 
To establish
\eqref{eq:par_Lewy_Stamp_bilateral}, we consider the auxiliary 
problem 
\begin{equation}
\label{eq:aux_EVI}
\left \{
~
\begin{aligned}
&\text{Find } \tilde y \in \KK_y \cap H^1(H^{-1})
\text{ such that } \tilde y(0) = y_0 \text{ and }
\\
&\int_0^T \left \langle \partial_t \tilde y  - \Delta \tilde y - 
\min(\partial_t \phi - \Delta \phi, u), v - \tilde y\right  \rangle_{H_0^1(\Omega)}\mathrm{d}t \geq 0 
~~\quad\forall v \in \KK_y, 
\end{aligned}
\right.
\end{equation}
where $\min(\cdot, \cdot)$ again acts by superposition 
and $\KK_y := \{v \in L^2(H_0^1) \mid y \leq v  \text{ a.e.\ in } Q\}$.
Due to 
\cref{prop:ex_uniqu}, our assumptions on $\phi$,
and the $L^2(H_0^1\cap H^2) \cap H^1(L^2)$-regularity
of $y$, we know that 
\eqref{eq:aux_EVI} 
has a unique solution $\tilde y \in L^2(H_0^1\cap H^2) \cap H^1(L^2)$.
From Stampacchia's lemma \cite[Theorem 5.8.2]{Attouch2006} 
and our assumptions on $\phi$, 
it further follows that 
$\min(\phi, \tilde y) \in \KK_y$.
This allows us to choose 
$\min(\phi, \tilde y)$ as the test function in 
\eqref{eq:aux_EVI} and to obtain that 
            \begin{align*}
                0 &\leq \int_0^T\langle \partial_t\tilde{y}-\Delta \tilde{y} - \min(\partial_t \phi - \Delta \phi, u), \min(\phi, \tilde{y})-\tilde{y}\rangle_{H_0^1(\Omega)} \dd t\\
                &=
                \left ( 
                \partial_t\tilde{y}-\Delta \tilde{y} - \min(\partial_t \phi - \Delta \phi, u), 
                \min(0, \phi - \tilde y)
                \right )_{L^2(L^2)}
                \\
                &\leq
                \left ( 
                \partial_t\tilde{y}-\Delta \tilde{y}, 
                \min(0, \phi - \tilde y)
                \right )_{L^2(L^2)}
                -
                \left ( \partial_t \phi - \Delta \phi, 
                \min(0, \phi - \tilde y)
                \right )_{L^2(L^2)}
                \\
                &=
                -\left ( 
                \partial_t (\tilde{y} -\phi) -\Delta (\tilde{y} - \phi), 
                \max(0,  \tilde y -  \phi)
                \right )_{L^2(L^2)}
                \\
                &= 
                 - \frac{1}{2}\|\max(0,  \tilde y(T) -  \phi(T))\|^2_{L^2(\Omega)}
                 - \|\nabla \max(0,  \tilde y -  \phi)\|^2_{L^2([L^2]^d)},
            \end{align*}
            where, in the last step, we have again used 
            the formulas for truncations from
            \cite[Lemma A.1]{ChristofVexler2020};
            cf.\ the proof of \cref{prop:ex_uniqu}. 
            Due to $\max(0,  \tilde y -  \phi) \in L^2(H_0^1)$
            and Poincaré's inequality, 
            the above implies  $\max(0,  \tilde y -  \phi) = 0$.
            In combination with $\tilde y \geq y \geq \psi$ a.e.\ in $Q$,
            this shows $\tilde y \in \KK_\psi^\phi$. 
            Using $\tilde y \in \KK_\psi^\phi$ as the test
            function in the EVI for $y$ 
            and $y$ as the test function in 
            the EVI for $\tilde y$,
            adding the resulting 
            inequalities,
            and exploiting that $\tilde y \geq y $
            holds 
            a.e.\ in $Q$
            now provides 
            us with 
            \begin{align*}
            0 &\geq \int_0^T ( u - \min(\partial_t\phi -\Delta \phi, u), y - \tilde{y} )_{L^2(\Omega)}\mathrm{d}t
            \\
            &\geq \int_0^T\langle \partial_ty - \partial_t \tilde{y}-\Delta y + \Delta \tilde{y}, y - \tilde{y} \rangle_{H_0^1(\Omega)}\mathrm{d}t
            \\
            & =
            \frac12  \| ( y - \tilde{y})(T) \|_{L^2(\Omega)}^2
            +
            \|\nabla ( y - \tilde{y})\|^2_{L^2([L^2]^d)}.
            \end{align*}
            Thus, $\tilde y = y$.  
            From the complementarity system \eqref{eq:unilateral_lower_obst_stat_sys}
            for \eqref{eq:aux_EVI}, we now obtain 
            \[
                0 \leq 
                \partial_t \tilde y  - \Delta \tilde y  - \min(\partial_t \phi - \Delta \phi, u) = 
                \partial_t  y  - \Delta y  - \min(\partial_t \phi - \Delta \phi, u)
                \text{ a.e.\ in } Q,
            \]
            and, consequently, 
            the first inequality 
            $\min(\partial_t \phi - \Delta \phi, u) \leq \partial_t  y  - \Delta y$ a.e.\ in $Q$ in \eqref{eq:par_Lewy_Stamp_bilateral}.
            The second inequality in \eqref{eq:par_Lewy_Stamp_bilateral}
            is obtained along the same lines.
            This proves the assertion for the case 
            $\smash{\KK = \KK_\psi^\phi}$. 

            To obtain the inequality 
            \eqref{eq:par_Lewy_Stamp_uni_upper}
            for the case $\KK = \KK^\phi$,
            it suffices to note that 
            the solution $y$ of \eqref{eq:bilateral_obst_weak} 
            with $\KK = \KK^\phi$ is trivially 
also the solution of the bilateral parabolic obstacle problem 
with lower obstacle $y$ and upper obstacle $\phi$.
(Note that $y$ and $\phi$
also satisfy the conditions on the obstacles 
in Assumption \ref{ass:standing:obstacle:prelims}.)
In view of \eqref{eq:par_Lewy_Stamp_bilateral}
and the 
inequality 
$\partial_t y  - \Delta y \leq  u$ a.e.\ in $Q$
obtained from \eqref{eq:unlateral_upper_obst_stat_sys}, 
this implies that 
    \begin{equation*}
       \begin{aligned}
        \min(\partial_t \phi -\Delta\phi, u)
        \leq 
         \partial_t y
            -\Delta y
            \leq\max(\partial_t y -\Delta y, u) 
            =u \text{ a.e.\ in } Q,
        \end{aligned}
        \end{equation*}
        which proves \eqref{eq:par_Lewy_Stamp_uni_upper}. 
The inequalities in \eqref{eq:par_Lewy_Stamp_uni_lower} 
are proved analogously. 
\end{proof}

As a consequence of \cref{lem:par_LewyStamp}, we obtain, for example, the following result.

\begin{corollary}[higher regularity for solutions]%
\label{cor:continuity_solutions}%
Suppose, in addition to our standing assumptions, 
that $y_0 \in H_0^1(\Omega) \cap H^2(\Omega)$  
and $\psi, \phi \in W^{1,\infty}(L^2) \cap L^\infty(H^2)$.
Assume further that $u \in L^\infty(L^2)$
and that $\KK \in \{\KK_\psi, \KK^\phi, \KK_\psi^\phi\}$.
Then the unique solution 
$y$ of the problem \eqref{eq:bilateral_obst_weak} 
satisfies $y \in C(\bar Q)$
and $u - \partial_t y + \Delta y \in L^\infty(L^2)$.
\end{corollary}

\begin{proof}
In the considered 
situation, we have 
$
 u,  \partial_t \phi -\Delta\phi , 
 \partial_t \psi -\Delta\psi  \in L^\infty(L^2).
$
Due to \cref{lem:par_LewyStamp},
this implies $u - \partial_t y + \Delta y \in L^\infty(L^2)$. In particular, 
$y$ solves a heat equation 
with homogeneous spatial Dirichlet boundary 
conditions,
a right-hand side in $L^\infty(L^2)$, 
and an initial value $y_0 \in H_0^1(\Omega) \cap H^2(\Omega)$. 
Invoking the regularity result
for the heat equation found in 
\cite[Theorem 3.1]{Rehberg2015}
now provides us with $y \in C(\bar Q)$.
\end{proof}

Having discussed the 
existence, uniqueness, and regularity 
of solutions of
\eqref{eq:bilateral_obst_weak}
for the cases $\KK = \KK_\psi$, 
$\KK = \KK^\phi$, and $\KK = \KK_\psi^\phi$,
we now turn our attention to 
the induced solution operators.
In addition to the functions  
that take $u$ to $y$, we also require 
the \emph{multiplier maps} associated with 
\eqref{eq:bilateral_obst_weak}
for our analysis.

\begin{definition}[solution operators and multiplier maps]%
\label{def:sol_ops}%
We denote by:
\begin{enumerate}[label=\roman*)]
\item $S_\psi$ the solution map 
$u \mapsto y$ of \eqref{eq:bilateral_obst_weak} 
in the case $\KK = \KK_\psi$;
\item  $S^\phi$ the solution map
$u \mapsto y$ of \eqref{eq:bilateral_obst_weak} 
in the case $\KK = \KK^\phi$;
\item  $S_\psi^\phi$ the solution map
$u \mapsto y$ of \eqref{eq:bilateral_obst_weak} 
in the case $\KK = \KK_\psi^\phi$;
\item 
$M_\psi$
the multiplier map defined by 
$M_\psi(u) := u - \partial_t S_\psi(u) + \Delta S_\psi(u)$;
\item 
$M^\phi$ 
the multiplier map defined by 
$M^\phi(u) := u - \partial_t S^\phi(u) + \Delta S^\phi(u)$;
\item 
$M_\psi^\phi$ 
the multiplier map defined by 
$M_\psi^\phi(u) := u - \partial_t S_\psi^\phi(u) + \Delta S_\psi^\phi(u)$.
\end{enumerate}
\end{definition}

Note that, by definition,
the maps
$M_\psi$, $M^\phi$, and 
\smash{$M_\psi^\phi$} reproduce precisely 
the slack variables 
$\lambda_\psi$, $\lambda^\phi$, and \smash{$\lambda_\psi + \lambda^\phi$}
in the complementarity systems 
\eqref{eq:unilateral_lower_obst_stat_sys}, 
\eqref{eq:unlateral_upper_obst_stat_sys}, 
and 
\eqref{eq:bilateral_obst_stat_sys}, 
respectively.
From 
 \cref{prop:ex_uniqu},
 we get that 
 the solution and multiplier maps  
are well defined as 
functions
$S_\psi, S^\phi, \smash{S_\psi^\phi}\colon
L^2(L^2) \to L^2(H_0^1\cap H^2) \cap H^1(L^2)$
and
$M_\psi, M^\phi, \smash{M_\psi^\phi}\colon
L^2(L^2) \to L^2(L^2)$, respectively. 
Under the regularity assumptions 
of \cref{cor:continuity_solutions},
we can further consider the solution and multiplier 
maps 
as operators 
$S_\psi, S^\phi, \smash{S_\psi^\phi}\colon
L^\infty(L^2) \to C(\bar Q)$
and 
$M_\psi, M^\phi, \smash{M_\psi^\phi}\colon
L^\infty(L^2) \to L^\infty(L^2)$. 
In what follows, it will always be clear 
from the context which domain and image spaces 
we choose for
the mappings in 
\cref{def:sol_ops}.

The next lemma shows that the functions 
$S_\psi, S^\phi$, and $S_\psi^\phi$
are globally Lipschitz.

\begin{lemma}[{Lipschitz continuity 
with values in $C(L^2) \cap L^2(H^1_0)$}]%
\label{lem:LipschitzPrimitiv}%
Suppose that 
$S \in \{S_\psi, S^\phi, S_\psi^\phi\}$.
Then $S$
is globally Lipschitz continuous
in the sense that there exists a 
constant $C = C(\Omega, T)>0$ satisfying  
    \begin{equation}
    \label{eq:LipEnergy}
    \begin{aligned}
        &
        \|S(u_1)-S(u_2)\|_{C(L^2)}
        + \|S(u_1)-S(u_2)\|_{L^2(H^1_0)}\\
        &\qquad \leq C \min(\|u_1-u_2\|_{L^2(H^{-1})},\|u_1-u_2\|_{L^1( L^2)})
        \qquad \forall u_1, u_2 \in L^2(L^2).
    \end{aligned}
    \end{equation}
\end{lemma}
\begin{proof}
The proof is completely analogous 
to that of \cite[Theorem 2.3]{Christof2019}.
(Recall that $H^1(L^2) \hookrightarrow C(L^2)$
so that the left-hand side of \eqref{eq:LipEnergy}
makes sense.)
\end{proof}

For the multiplier maps, we have the following 
Lipschitz continuity result
that already hints at
how \eqref{eq:bilateral_obst_weak} 
is connected to the analysis of 
\cref{sec:3}; cf.\ \eqref{eq:Lipschitz_estimate_abstract}.

\begin{lemma}[global 1-Lipschitz continuity of the multiplier maps]%
\label{lem:mult_contractive}%
Suppose that $(S, M) \in 
\{(S_\psi, M_\psi),(S^\phi, M^\phi),
(S_\psi^\phi, M_\psi^\phi)\}$. 
Then, for all 
$u_1, u_2 \in L^2(L^2)$, we have
\begin{equation}
\label{eq:M_Lip}
\|
M(u_1) - M(u_2)
\|_{L^1(L^1)}
 +
 \|
S(u_1)(T) - S(u_2)(T)
\|_{L^1(\Omega)}
\leq
\|
u_1 - u_2
\|_{L^1(L^1)}.
\end{equation}
\end{lemma}

\begin{proof}
We prove the assertion for
$\smash{(S,M) = (S_\psi^\phi, M_\psi^\phi)}$.
The other cases are again analogous. 
Let $u_1, u_2\in L^2(L^2)$ be arbitrary but fixed, 
define $y_i := S_\psi^\phi(u_i)$ and 
$\smash{\lambda_i := M_\psi^\phi(u_i)}$
for $i=1,2$,
and denote by 
$\lambda_{i,\psi}$
and
$\smash{\lambda_{i}^\phi}$
the two parts of $\lambda_i$ in \eqref{eq:bilateral_obst_stat_sys}. 
We consider for $\varepsilon > 0$ the mollified signum function
\[
\rho_\varepsilon\colon \R \to \R,
\qquad
\rho_\varepsilon(s) := \max\left (-1, \min\left (1, \frac{s}{\varepsilon} \right ) \right ),
\]
and its antiderivative
\[
R_\varepsilon\colon \R \to \R,
\qquad 
R_\varepsilon(s) 
:=
\begin{cases}
\displaystyle
\frac{1}{2\varepsilon} s^2 &\text{ if } |s| < \varepsilon,
\\[0.15cm]
|s| - \frac{\varepsilon}{2}
&\text{ if } |s| \geq \varepsilon.
\end{cases}
\]
Note that $\rho_\varepsilon$
is globally Lipschitz, bounded, 
and non-decreasing with $\rho_\varepsilon(0) = 0$.
From \cite[Lemma A.1]{ChristofVexler2020},
we further get that 
$\rho_{\varepsilon}(y_1-y_2) \in L^2(H_0^1)$.
If we combine all of this 
with the identities 
\begin{align*}
    \partial_t y_1 - \Delta y_1 = u_1 - \lambda_1
    \text{ a.e.\ in } Q
    \qquad \text{and} \qquad
    \partial_t y_2 - \Delta y_2 = u_2 - \lambda_2
    \text{ a.e.\ in } Q
\end{align*}
obtained from \cref{prop:ex_uniqu}
and again use the formulas for derivatives of
truncations in 
\cite[Lemma A.1]{ChristofVexler2020},
then it follows that 
\begin{equation}
\label{eq:randomeq374b}
\begin{aligned}
&\int_\Omega R_\varepsilon(y_1(T) - y_2(T)) \dd x
+
\int_{\{y_1\neq y_2\}}
(\lambda_1-\lambda_2) \rho_{\varepsilon}(y_1-y_2)
\dd (t,x)
\\
&\leq
\int_0^T \left ( \partial_t (y_1-y_2) + (\lambda_1-\lambda_2),\rho_{\varepsilon}(y_1-y_2)\right )_{L^2(\Omega)}
+ \int_\Omega \rho_\varepsilon'(y_1 - y_2)
|\nabla (y_1 - y_2)|^2 \dd x 
\dd t \\
    &=\int_0^T \left ( \partial_t (y_1-y_2) - \Delta (y_1-y_2)+ (\lambda_1-\lambda_2),\rho_{\varepsilon}(y_1-y_2)\right )_{L^2(\Omega)}\dd t \\
    &=\int_0^T\int_\Omega(u_1-u_2)\rho_{\varepsilon}(y_1-y_2) \dd x \dd t\\
    &\leq \int_{\{y_1\neq y_2\}}
    | u_1-u_2 | \dd (t,x).
\end{aligned}
\end{equation}
Here, we have used 
Hölder's inequality 
in the 
last line and 
$\{y_1\neq y_2\}$
is short for 
$\{(t,x) \in Q \mid y_1(t,x)\neq y_2(t,x) \}$,
defined up to sets of measure zero. 
Letting $0 < \varepsilon \to 0$ in 
\eqref{eq:randomeq374b},
exploiting that $0 \leq R_\varepsilon(s) \leq |s|$ holds for all $s \in \R$,
and invoking the dominated convergence theorem
leads to
\begin{equation}
\label{eq:randomeq2673ge83n}
\begin{aligned}
    \| y_1(T) - y_2(T) \|_{L^1(\Omega)}
+
\int_{\{y_1\neq y_2\}}
&( \lambda_1-\lambda_2) \sgn(y_1-y_2)\dd (t,x)
\\
     &\qquad\leq \|u_1-u_2\|_{L^1(\{y_1\neq y_2\})}.
     \end{aligned}
\end{equation}
From the
complementarity system \cref{eq:bilateral_obst_stat_sys}
and a simple pointwise distinction of cases, 
we further obtain that, a.e.\ on the set $\{y_1\neq y_2\}$, we have
\begin{align*}
    \lambda_1(t,x)-\lambda_2(t,x) 
    =
    \begin{cases}
        \lambda_1^\phi(t,x)-\lambda_{2,\psi}(t,x)
        \geq 0 \ \text{ if } \ y_2(t,x) < y_1(t,x),\\
        \lambda_{1,\psi}(t,x)-\lambda_{2}^\phi(t,x) \leq 0 \ \text{ if } \ y_1(t,x) < y_2(t,x).
    \end{cases}
\end{align*}
Thus, $\sgn(y_1 - y_2) = \sgn(\lambda_1 - \lambda_2)$
a.e.\ on $\{y_1\neq y_2\} \cap \{\lambda_1 \neq \lambda_2\}$, and 
we obtain from \eqref{eq:randomeq2673ge83n} that
\begin{equation}
\label{eq:randomeq2673ge83n-2}
    \| y_1(T) - y_2(T) \|_{L^1(\Omega)}
+
\|\lambda_1-\lambda_2\|_{L^1(\{y_1\neq y_2\})}
     \leq \|u_1-u_2\|_{L^1(\{y_1\neq y_2\})}.
\end{equation}
It remains to study 
$\lambda_1 - \lambda_2$ 
on the set $\{y_1 = y_2\} :=  \{(t,x) \in Q \mid y_1(t,x) = y_2(t,x) \}$ (again defined up to sets of measure zero). 
To this end, we note that the regularity 
$y_1, y_2 \in L^2(H_0^1\cap H^2) \cap H^1(L^2)$
and the formulas 
in \cite[Lemma A.1]{ChristofVexler2020} imply that 
\[
    \partial_t y_1 - \Delta y_1 
    =
    \partial_t y_2 - \Delta y_2 
    \quad
    \text{ a.e.\ on } \{y_1 = y_2\}.
\]
In view of the complementarity system \cref{eq:bilateral_obst_stat_sys},
this entails
\[
    \lambda_1 - \lambda_2
    =
    u_1 - \left (\partial_t y_1 - \Delta y_1 \right )
    -
    u_2 + \left (\partial_t y_2 - \Delta y_2 \right )
    =
    u_1 - u_2\quad
    \text{ a.e.\ on } \{y_1 = y_2\}.
\]
 We thus have $\lambda_1 - \lambda_2 = u_1 - u_2$
 a.e.\ on $\{y_1 = y_2\}$
 and, consequently, 
 \[
 \|\lambda_1-\lambda_2\|_{L^1(\{y_1 = y_2\})}
     = \|u_1-u_2\|_{L^1(\{y_1 = y_2\})}.
 \]
 Combining the above with 
 \eqref{eq:randomeq2673ge83n-2} yields 
 \eqref{eq:M_Lip} for $\smash{(S,M) = (S_\psi^\phi, M_\psi^\phi)}$
 as desired.  
\end{proof}

As a  byproduct of 
\cref{lem:mult_contractive}, we also obtain the 
following result.

\begin{corollary}[{global 1-Lipschitz continuity
with values in $C(L^1)$}]%
\label{cor:LipCL1}%
Suppose that 
$S \in \{S_\psi, S^\phi, S_\psi^\phi\}$.
Then it holds
    \begin{equation}
    \label{eq:LipCL1}
    \begin{aligned}
        &\|S(u_1)-S(u_2)\|_{C(L^1)} 
        \leq \|u_1-u_2\|_{L^1( L^1)}
        \qquad \forall u_1, u_2 \in L^2(L^2).
    \end{aligned}
    \end{equation}
\end{corollary}
\begin{proof}
By considering
\eqref{eq:bilateral_obst_weak}
on the interval $(0,s)$ instead of $(0,T)$ for 
$0 < s \leq T$ and applying 
\cref{lem:mult_contractive} to 
the resulting restricted EVI, one readily obtains that
\begin{equation*}
\|
S(u_1)(s) - S(u_2)(s)
\|_{L^1(\Omega)}
\leq
\|
u_1 - u_2\|_{L^1(0,s;L^1(\Omega))}
\leq
\|
u_1 - u_2\|_{L^1(0,T;L^1(\Omega))}
\end{equation*}
holds for all $u_1, u_2 \in L^2(0,T;L^2(\Omega))$
and all $s \in (0,T]$.
Taking the supremum over all $s \in (0,T]$
on the left-hand side of this estimate 
yields \eqref{eq:LipCL1} as desired.
\end{proof}

In the situation of \cref{cor:continuity_solutions},
in which the solutions of 
\eqref{eq:bilateral_obst_weak} are in 
$C(\bar Q)$, 
the maps $S_\psi$, $S^\phi$, and $S_\psi^\phi$
are also Lipschitz 
with values in the continuous functions.

\begin{lemma}[Lipschitz continuity with values in $C(\bar Q)$]
\label{lem:Lipschitz}
Suppose, in addition to our standing assumptions, 
that $y_0 \in H_0^1(\Omega) \cap H^2(\Omega)$  
and $\psi, \phi \in W^{1,\infty}(L^2) \cap L^\infty(H^2)$.
Let $S \in \{S_\psi, S^\phi, S_\psi^\phi\}$. 
Then there exists a constant 
$\hat C = \hat C(\Omega, T, d) > 0$ satisfying 
 \begin{equation}
 \label{eq:LipschitzEstimate}
 \|
S(u_1) - S(u_2)
\|_{C(\bar Q)}
\leq 
\hat C
\| 
u_1 - u_2 
\|_{L^\infty(L^2)}
\qquad 
\forall u_1, u_2 \in L^\infty(L^2).
 \end{equation}
\end{lemma}

\begin{proof}
We prove \eqref{eq:LipschitzEstimate} only for 
\smash{$S = S_\psi^\phi$}. 
The unilateral cases are completely analogous. 
Let $u_1, u_2 \in L^\infty(L^2)$
be given and define $y_1 := \smash{S_\psi^\phi}(u_1)$ and $y_2 := \smash{S_\psi^\phi}(u_2)$. 
We consider for $l \geq 0$ the family of truncations 
\[
(y_1 - y_2)_l := 
\min(0, y_1 - y_2 + l) + \max(0, y_1 - y_2 - l),
\]
where $\max(0, \cdot)$ and $\min(0, \cdot)$
again act by superposition. 
Due to Stampacchia's lemma \cite[Theorem 5.8.2]{Attouch2006},
we have $(y_1 - y_2)_l \in L^2(H_0^1)$ for all $l \geq 0$.
A simple distinction of cases further shows 
$y_1 - (y_1 - y_2)_l,
y_2 + (y_1 - y_2)_l \in \KK_\psi^\phi$ for all $l\geq 0$. 
This allows us to choose, 
for every $s \in (0, T]$,
$y_1 - \mathds{1}_{[0,s]}(y_1 - y_2)_l$ as the test function 
in the EVI for $y_1$ and 
$y_2 + \mathds{1}_{[0,s]}(y_1 - y_2)_l$ as the test function 
in the EVI for $y_2$ 
and to add the resulting inequalities to obtain 
\begin{align}
\label{eq:randome2273h383h}
\int_0^s \left \langle \partial_t(y_2 -y_1)  
- \Delta (y_2 -y_1) + u_1 - u_2, 
(y_1 - y_2)_l\right  \rangle_{H_0^1(\Omega)}\mathrm{d}t \geq 0.
\end{align}
By exploiting the formulas in \cite[Lemma A.1]{ChristofVexler2020},
\eqref{eq:randome2273h383h} can be rewritten as 
\begin{align*}
\frac{1}{2} \| (y_1 -y_2)_l(s)\|_{L^2(\Omega)}^2
+
\int_0^s 
\| \nabla (y_1 - y_2)_l \|_{L^2(\Omega)^d}^2
\mathrm{d}t
\leq
\int_0^s \left (u_1 - u_2, 
(y_1 - y_2)_l\right )_{L^2(\Omega)}\mathrm{d}t
\end{align*}
for all $s \in (0, T]$ and $l \geq 0$.
The above implies that 
\begin{align*}
\| (y_1 -y_2)_l\|_{L^\infty(L^2)}^2
+
\| \nabla (y_1 -y_2)_l\|_{L^2([L^2]^d)}^2
\leq
3
\int_0^T
\int_\Omega  |u_1 - u_2 |
\,
 | (y_1 - y_2)_l| \dd x \mathrm{d}t
\end{align*}
holds for all $l\geq 0$ and, 
in view of Poincaré's inequality and 
\cite[Lemma A.2]{ChristofVexler2020},
that there exists a constant $\hat C = \hat C(\Omega, T, d) > 0$ satisfying 
$
\| y_1 - y_2 \|_{L^\infty(Q)}
 \leq
 \hat C 
\| u_1 - u_2\|_{L^\infty(L^2)}
$.
This proves \eqref{eq:LipschitzEstimate} 
for \smash{$S = S_\psi^\phi$} as desired.
\end{proof}

For our analysis of \eqref{eq:bilateral_obst_intro}, 
we also have to study 
situations
in which the obstacles 
take the values $\pm \infty$. 
This case is addressed 
in the following lemma.

\begin{lemma}[EVIs with singular obstacles]%
\label{lem_deg_obstacle_probs}%
Suppose, in addition to our standing assumptions, 
that $y_0 \in H_0^1(\Omega) \cap H^2(\Omega)$  
and $\psi, \phi \in W^{1,\infty}(L^2) \cap L^\infty(H^2)$.
Assume further 
that $V_\psi$ and  $V^\phi$ are 
bounded open subsets of 
$\R^{d+1}$ (possibly empty) and that 
$\chi_\psi\colon \R^{d+1}\to [-\infty, 0]$ and
$\chi^\phi\colon \R^{d+1} \to [0, \infty]$ are maps
with the following properties:
\begin{enumerate}[label=\alph*)]
\item 
$\chi_\psi$ is real-valued on $V_\psi$
with 
$\chi_\psi|_{V_\psi} \in C^2(V_\psi)$,
it holds $\chi_\psi \equiv - \infty$ on $\R^{d+1} \setminus V_\psi$, and $\chi_\psi\colon \R^{d+1} \to [-\infty, 0]$
is upper semicontinuous;
\item 
$\chi^\phi$ is real-valued on $V^\phi$ with
$\chi^\phi|_{V^\phi} \in C^2(V^\phi)$,
it holds $\chi^\phi \equiv  \infty$ on $\R^{d+1} \setminus V^\phi$,
and $\chi^\phi\colon \R^{d+1} \to [0, \infty]$
is lower semicontinuous.
\end{enumerate}
Define 
\begin{align}
\label{eq:hat_obstacles}
\hat \psi(t,x) := \psi(t,x) + \chi_\psi(t,x)
\qquad\text{and}\qquad
\hat \phi(t,x) := \phi(t,x) + \chi^\phi(t,x)
\end{align}
for a.a.\ $(t,x) \in Q$
and
 $\KK_{\hat \psi}^{\hat \phi} 
:=
 \{v \in L^2(H_0^1) \mid
\hat \psi \leq v \leq \hat \phi \text{ a.e.\ in } Q
 \}.$
Then:
\begin{enumerate}[label=\roman*)]
\item\label{lem_deg_obstacle_probs:i}
For every $u \in L^\infty(L^2)$,
there exists a unique solution $y$ of the problem
\begin{equation}
\label{eq:bilateral_obst_weak_deg}
\left \{
~
\begin{aligned}
&\text{Find } y \in \KK_{\hat \psi}^{\hat \phi}  \cap H^1(H^{-1})
\text{ such that } y(0) = y_0 \text{ and }
\\
&\int_0^T \left \langle \partial_t y  - \Delta y - u, v - y\right  \rangle_{H_0^1(\Omega)}\mathrm{d}t \geq 0 
~\quad\forall v \in \KK_{\hat \psi}^{\hat \phi}. 
\end{aligned}
\right.
\end{equation}
\item\label{lem_deg_obstacle_probs:ii}
A function $y$ solves 
\eqref{eq:bilateral_obst_weak_deg} 
for some $u \in L^\infty(L^2)$
if and only if there exist multipliers 
$\lambda_{\hat \psi}$ and $\lambda^{\hat \phi}$
such that the following complementarity system holds:%
\begin{equation}
\label{eq:bilateral_obst_stat_sys_deg}
\left \{
~
\begin{gathered}
y \in L^2(H_0^1\cap H^2) \cap H^1(L^2),
\quad 
\lambda_{\hat \psi}, \lambda^{\hat \phi} \in L^2(L^2),
\quad y(0) = y_0, 
\\
\partial_t y  - \Delta y  = 
u - \lambda_{\hat \psi} - \lambda^{\hat \phi}\quad \text{ a.e.\ in }Q,
\\
\lambda_{\hat \psi} \leq 0,
\quad
y - \hat \psi \geq 0,
\quad
\lambda_{\hat \psi}(y - \hat \psi) = 0
\quad \text{ a.e.\ in }V_\psi \cap Q,
\\
\,\lambda^{\hat \phi} \geq 0,
\quad
y - \hat \phi \leq 0,
\quad
\lambda^{\hat \phi}(y - \hat \phi) = 0
\quad\text{ a.e.\ in }V^\phi \cap Q,
\\
\lambda_{\hat \psi} = 0~\text{ a.e.\ in }Q \setminus V_\psi,
\qquad
\lambda^{\hat \phi} = 0~\text{ a.e.\ in }Q \setminus V^\phi.
\end{gathered}
\right.
\end{equation}
\item\label{lem_deg_obstacle_probs:iii}
If
$y$ solves 
\eqref{eq:bilateral_obst_weak_deg}
for $u \in L^\infty(L^2)$, then 
$y \in C(\bar Q)$ and $u - \partial_t y + \Delta y 
\in L^\infty(L^2)$.
\item\label{lem_deg_obstacle_probs:iv}
If we define 
\[
S_{\hat \psi}^{\hat \phi}\colon L^\infty(L^2) \to L^2(H_0^1\cap H^2) \cap H^1(L^2) \cap C(\bar Q), 
\qquad u \mapsto y,
\]
to be the solution operator
of \eqref{eq:bilateral_obst_weak_deg}
and denote by 
\[
M_{\hat \psi}^{\hat \phi}\colon L^\infty(L^2) \to L^\infty(L^2),
\qquad 
M_{\hat \psi}^{\hat \phi}(u) := 
 u - \partial_t S_{\hat \psi}^{\hat \phi}(u) + 
 \Delta S_{\hat \psi}^{\hat \phi}(u),
 \]
 the associated multiplier map,
 then these functions 
satisfy 
 the same Lipschitz estimates 
 as in the regular case, i.e., 
 \eqref{eq:LipEnergy}, 
 \eqref{eq:M_Lip}, 
 \eqref{eq:LipCL1},
 and 
 \eqref{eq:LipschitzEstimate}.
\end{enumerate}
\end{lemma}

\begin{proof}
That the EVI \eqref{eq:bilateral_obst_weak_deg}
can have at most one solution 
follows from the same arguments
as in the proof of \cref{prop:ex_uniqu}. 

To prove that \eqref{eq:bilateral_obst_weak_deg} 
possesses a solution, we suppose that $r>0$ 
is a given and fixed number and consider the function  
$w := S_\psi^\phi(0) \in 
C(\bar Q) \cap L^2(H_0^1 \cap H^2) \cap H^1(L^2)$.
From
\cref{cor:continuity_solutions},
we get that 
$\partial_t w - \Delta w \in L^\infty(L^2)$.
This allows us to define, 
using the 
constant 
$\hat C = \hat C(\Omega, T, d)$ from \cref{lem:Lipschitz}, 
\[
C_r := 
\hat C \left (\|\partial_t w - \Delta w\|_{L^\infty(L^2)}
+
  r \right ) \in (0, \infty).
\]
Recall that 
upper semicontinuity implies that 
preimages of intervals $[c, \infty]$, $c \in \R$,
are closed. For $\chi_\psi$,
this entails that
$E_\psi := (\chi_\psi)^{-1}([- C_r - 1, \infty]) \subset \R^{d+1}$
is a compact subset of $V_\psi$. 
Analogously, we obtain that 
$E^\phi := (\chi^\phi)^{-1}([-\infty, C_r + 1]) \subset \R^{d+1}$
is a compact subset of $V^\phi$. 
Consider now functions 
$\rho_\psi, \rho^\phi$ satisfying 
\begin{gather*}
\rho_\psi \in C_c^\infty(\R^{d+1}),
\qquad 0 \leq \rho_\psi \leq 1
\text{ in } \R^{d+1},
\qquad
\supp(\rho_\psi) \subset V_\psi,
\qquad
 \rho_\psi \equiv 1 \text{ on }
 E_\psi,
\\
 \rho^\phi\in C_c^\infty(\R^{d+1}),
\qquad 0 \leq \rho^\phi \leq 1
\text{ in } \R^{d+1},
\qquad
\supp(\rho^\phi) \subset V^\phi,
\qquad 
 \rho^\phi \equiv 1 \text{ on }
E^\phi,
\end{gather*}
and define\\[-0.7cm]
\begin{align*}
\bar \psi
&:=
\begin{cases}
\psi + \rho_\psi\chi_\psi
+ (1 - \rho_\psi)(-C_r - 1)  & \text{a.e.\ in } V_\psi \cap Q,
\\
\psi -C_r - 1 &\text{a.e.\ in }Q \setminus V_\psi,
\end{cases}
\\
\bar \phi
&:=
\begin{cases}
\phi + \rho^\phi \chi^\phi + (1 - \rho^\phi)(C_r + 1)  &~~~\text{a.e.\ in } V^\phi \cap Q,
\\
\phi + C_r + 1 &~~~\text{a.e.\ in }Q \setminus V^\phi.
\end{cases}
\end{align*}
Then it follows from our assumptions 
on 
$\psi$, $\phi$, $\chi_\psi$, $\chi^\phi$,
$\rho_\psi$, and $\rho^\phi$ 
and the definitions of $E_\psi$ and $E^\phi$ 
that 
$\bar \psi, \bar \phi \in 
W^{1,\infty}(L^2) \cap L^\infty(H^2)$ and 
\begin{equation}
\label{eq:tilde_obs_props}
\begin{gathered}
\bar \psi
\begin{cases}
= \psi + \chi_\psi = \hat \psi \leq \psi& 
\hspace{-0.15cm}\text{a.e.\ in } 
E_\psi \cap Q,
\\
\leq \psi - C_r - 1
&\hspace{-0.15cm}\text{a.e.\ in }Q \setminus E_\psi,
\end{cases}
\qquad 
\bar \phi
\begin{cases}
= \phi + \chi^\phi = \hat \phi \geq \phi& 
\hspace{-0.15cm}\text{a.e.\ in } 
E^\phi \cap Q,
\\
\geq \phi + C_r + 1
&\hspace{-0.15cm}\text{a.e.\ in }Q \setminus E^\phi.
\end{cases}
\end{gathered}
\end{equation}
Assume now that 
some $u \in L^\infty(L^2)$
satisfying $\|u\|_{L^\infty(L^2)} \leq r$
is given. Then it follows 
from \cref{prop:ex_uniqu}
and \cref{cor:continuity_solutions}
that the bilateral parabolic 
obstacle problem with 
lower obstacle 
$\bar \psi$ and 
upper obstacle $\bar \phi$
and parameter $u$ has 
a unique solution 
\mbox{$y \in L^2(H^2)\cap H^1(L^2) \cap C(\bar Q)$}.
Due to $w \in \KK_\psi^\phi$ and the 
inequalities 
\smash{$\bar \psi \leq \psi \leq \phi \leq \bar \phi$}
obtained from \eqref{eq:tilde_obs_props},
we further get that the 
same obstacle problem with parameter 
$\partial_t w - \Delta w$ is solved by $w$. In view of 
\eqref{eq:LipschitzEstimate} and the triangle inequality,
this yields
$
\|y - w\|_{C(\bar Q)}
\leq
\hat C \left(\|\partial_t w - \Delta w\|_{L^\infty(L^2)} + r\right ) = C_r
$,
which, due to
$\psi \leq w \leq \phi$, entails 
\begin{equation}
\label{eq:randomeq82hei3n3}
\psi - C_r \leq y \leq \phi + C_r \text{ a.e.\ in } Q.
\end{equation}
Because of \eqref{eq:tilde_obs_props}, 
\eqref{eq:randomeq82hei3n3}
implies $y \geq \bar \psi +1$  
a.e.\ in $Q \setminus E_\psi$ 
and 
$y \leq \bar \phi -1$ 
a.e.\ in $Q \setminus E^\phi$.
From \eqref{eq:randomeq82hei3n3}
and
the definitions of 
$\hat \psi $,
$\hat \phi $,
$E_\psi$,
and $E^\phi$,
we also 
obtain that 
$y \geq \hat \psi + 1$
a.e.\ in $Q \setminus E_\psi$ 
and 
that $y \leq \hat \phi - 1$
a.e.\ in $Q \setminus E^\phi$. 
In combination 
with the 
identities 
$\bar \psi = \hat \psi$
a.e.\ on $E_\psi \cap Q$
and 
$\bar \phi = \hat \phi$
a.e.\ on $E^\phi \cap Q$
in \eqref{eq:tilde_obs_props}, 
all of this implies that the 
complementarity system 
\begin{equation*}
\begin{gathered}
y \in L^2(H_0^1\cap H^2) \cap H^1(L^2),
\quad 
\lambda_{\bar \psi}, \lambda^{\bar \phi} \in L^2(L^2),
\quad y(0) = y_0, 
\\
\partial_t y  - \Delta y  = 
u - \lambda_{\bar \psi} - \lambda^{\bar \phi}\quad \text{ a.e.\ in }Q,
\\
\lambda_{\bar \psi} \leq 0,
\qquad  
y - \bar \psi \geq 0,
\qquad
\lambda_{\bar \psi}(y - \bar \psi) = 0\quad \text{ a.e.\ in }Q,
\\
\lambda^{\bar \phi} \geq 0,
\qquad  
y - \bar \phi \leq 0,
\qquad
\lambda^{\bar \phi}(y - \bar \phi) = 0\quad \text{ a.e.\ in }Q,
\end{gathered}
\end{equation*}
which $y$ satisfies 
as a solution of the bilateral parabolic 
obstacle problem with obstacles
$\bar \psi$ and 
$\bar \phi$, 
reduces precisely to \eqref{eq:bilateral_obst_stat_sys_deg}
(with $\lambda_{\hat \psi} = \lambda_{\bar \psi}$
and $\lambda^{\hat \phi} = \lambda^{\bar \phi}$).
This shows that \eqref{eq:bilateral_obst_stat_sys_deg}
possesses a solution 
\smash{$(y, \lambda_{\hat \psi}, \lambda^{\hat \phi})$}.
Using the exact same arguments 
as in the proof of 
\cref{prop:ex_uniqu}, 
one now checks that, 
if $(y, \lambda_{\hat \psi}, \lambda^{\hat \phi})$
solves \eqref{eq:bilateral_obst_stat_sys_deg},
then $y$ solves \eqref{eq:bilateral_obst_weak_deg}. 
In combination with the 
already established fact that \eqref{eq:bilateral_obst_weak_deg}
can have at most one solution, 
this proves the assertions
in points \ref{lem_deg_obstacle_probs:i}
and \ref{lem_deg_obstacle_probs:ii}
of the lemma. 

To 
obtain \ref{lem_deg_obstacle_probs:iii}
and \ref{lem_deg_obstacle_probs:iv},
we note that, in all of the above,
$u \in L^\infty(L^2)$
was an arbitrary function 
with $\|u\|_{L^\infty(L^2)} \leq r$.
This means that we have shown that 
\begin{equation}
\label{eq:transfer}
    S_{\hat \psi}^{\hat \phi}(u)
    =
    S_{\bar \psi}^{\bar \phi}(u)
    \qquad 
    \forall u \in B_r^{L^\infty(L^2)}(0)
\end{equation}
holds, where $S_{\hat \psi}^{\hat \phi}$
is the solution map of \eqref{eq:bilateral_obst_weak_deg} 
and 
$ S_{\bar\psi}^{\bar \phi}$
the solution map 
of the bilateral obstacle problem 
with the obstacles
in \eqref{eq:tilde_obs_props}. 
That the solutions 
of \eqref{eq:bilateral_obst_weak_deg} 
satisfy 
$y \in C(\bar Q)$ 
and
$u - \partial_t y + \Delta y 
\in L^\infty(L^2)$
for all $u \in L^\infty(L^2)$
with $\|u\|_{L^\infty(L^2)} \leq r$
and that the solution and multiplier map 
of  \eqref{eq:bilateral_obst_weak_deg} 
satisfy the 
Lipschitz estimates 
\eqref{eq:LipEnergy}, 
 \eqref{eq:M_Lip}, 
 \eqref{eq:LipCL1},
  and 
 \eqref{eq:LipschitzEstimate}
for all $u_1, u_2 \in L^\infty(L^2)$
with 
$\|u_i\|_{L^\infty(L^2)} \leq r$, $i=1,2$,
follows immediately from \eqref{eq:transfer}. 
As $r > 0$ was arbitrary, 
we obtain the assertions 
in \ref{lem_deg_obstacle_probs:iii}
and \ref{lem_deg_obstacle_probs:iv} 
also for all $u, u_1, u_2 \in L^\infty(L^2)$.
(Note that the Lipschitz constants in 
\cref{lem:LipschitzPrimitiv,lem:mult_contractive,lem:Lipschitz,cor:LipCL1} 
are independent of the obstacles so that
it does not matter here that 
$\bar \psi$ and $\bar \phi$ depend on $r$.)
This completes the proof.
\end{proof}

We emphasize that the 
choices $V^\phi = \emptyset$ and 
$V_\psi = \emptyset$ are allowed in 
\cref{lem_deg_obstacle_probs}. For these 
sets, \eqref{eq:bilateral_obst_weak_deg} 
reduces to a unilateral obstacle problem 
(or the heat equation, respectively, if 
$V^\phi = V_\psi = \emptyset$).
If one of the sets $V^\phi$, $V_\psi$
is empty, then we can proceed along the lines 
of \cite{Christof2019} to 
obtain directional differentiability results for the 
solution operator and the multiplier map 
of \eqref{eq:bilateral_obst_weak_deg}.

\begin{lemma}[directional derivatives for unilateral 
parabolic obstacle problems]%
\label{lem:dir_diff_uni}%
Consider the situation 
in \cref{lem_deg_obstacle_probs}
in the special case 
$V^\phi = \emptyset$ and 
$\chi^\phi \equiv \infty$.
Denote the solution map of the
(in this case unilateral) parabolic
obstacle problem \eqref{eq:bilateral_obst_weak_deg} by 
\[
    S_{\hat \psi} \colon L^\infty(L^2) \to L^2(H_0^1\cap H^2) \cap H^1(L^2) \cap C(\bar Q),
    \qquad 
    u \mapsto y, 
\]
and the associated multiplier map by 
\[
M_{\hat \psi}\colon L^\infty(L^2) \to L^\infty(L^2),
\qquad 
M_{\hat \psi}(u) := 
 u - \partial_t S_{\hat \psi}(u) + 
 \Delta S_{\hat \psi}(u).
 \]
 Then 
 $S_{\hat \psi}$ and 
 $M_{\hat \psi}$
 can be extended uniquely to
 globally 1-Lipschitz continuous maps
\[
    S_{\hat \psi}\colon L^1(L^1) \to C(L^1)
    \qquad
    \text{and}
    \qquad
    M_{\hat \psi} \colon L^1(L^1) \to L^1(L^1).
\] 
The obtained extensions 
satisfy
\begin{equation}
\begin{aligned}
\label{eq:M_Lip_1}
 \|
M_{\hat \psi}(u_1) - M_{\hat \psi}(u_2)
\|_{L^1(L^1)}
&+
 \|
S_{\hat \psi}(u_1)(T) - S_{\hat \psi}(u_2)(T)
 \|_{L^1(\Omega)}
\\
&\leq
 \|
u_1 - u_2  \|_{L^1(L^1)}
\qquad \forall u_1, u_2 \in L^1(L^1)
\end{aligned}
\end{equation}
and are Hadamard directionally differentiable 
in the sense that,
for all $u, g \in L^1(L^1)$
and $s \in [0,T]$, 
there exist
\smash{$M_{\hat \psi}'(u;g) \in  \MM((0,T] \times \Omega)$}
and 
\smash{$S_{\hat \psi}'(u;g)(s) \in  L^1(\Omega)$}
with%
\begin{equation}
\label{eq:had_probs_unilatera}
\begin{aligned}
    &\{g_n\} \subset L^1(L^1), 
    \{\tau_n\} \subset \left (0, \infty \right ),
    g_n \to g \text{ in } L^1(L^1), 
    \tau_n \to 0
    \\
    &
    \qquad
    \Rightarrow
    ~
    \left \{
    \begin{aligned}
    &\frac{M_{\hat \psi}(u + \tau_n g_n) - M_{\hat \psi}(u)}{\tau_n}
    \weaklystar M_{\hat \psi}'(u;g) &&\text{in } \MM((0,T] \times \Omega),
    \\
    &\frac{S_{\hat \psi}(u + \tau_n g_n)(s) - S_{\hat \psi}(u)(s)}{\tau_n}
    \to S_{\hat \psi}'(u;g)(s) &&\text{in }L^1(\Omega). 
    \end{aligned}
    \right.
\end{aligned}
\end{equation}
\end{lemma}

\begin{proof}
That 
$S_{\hat \psi}$ and $M_{\hat \psi}$
admit unique extensions of the asserted type 
follows from the fact that 
the Lipschitz estimates 
\eqref{eq:M_Lip} 
and
\eqref{eq:LipCL1}
hold for $S_{\hat \psi}$ and $M_{\hat \psi}$
on $L^\infty(L^2)$
by \cref{lem_deg_obstacle_probs},
the density of $L^\infty(L^2)$ in $L^1(L^1)$,
the completeness of the spaces 
$C(L^1)$ and 
$L^1(L^1)$,
and 
\cite[Proposition 1.2.3]{Weaver1999}.
To obtain \eqref{eq:M_Lip_1},
one can simply extend \eqref{eq:M_Lip}
for $S_{\hat \psi}$ and $M_{\hat \psi}$
from 
$L^\infty(L^2)$ to $L^1(L^1)$ by continuity. 
It remains to prove the 
Hadamard 
directional 
differentiability properties in \eqref{eq:had_probs_unilatera}.

To this end, we first establish 
along the lines of the proof 
of \cite[Lemma 3.3]{ChristofWachsmuth2023} 
that the function  
$S_{\hat \psi}$
is pointwise-a.e.\ convex. 
Suppose that $u_1, u_2 \in L^\infty(L^2)$ and 
$\alpha \in [0,1]$ are given
and define 
$y_1 := S_{\hat \psi}(u_1)$,
$y_2 := S_{\hat \psi}(u_2)$,
$y_{12} := S_{\hat \psi}(\alpha u_1 + (1 - \alpha)u_2)$,
and 
$w := \alpha y_1 + (1 - \alpha)y_2 - y_{12}$.
Then 
\cite[Lemma A.1]{ChristofVexler2020} implies 
$\min(0,w) \in L^2(H_0^1) \cap H^1(L^2)$
and a simple distinction of cases shows that 
$y_1 - \min(0, w) \geq \hat \psi$,
$y_2 - \min(0, w) \geq \hat \psi$, and
$y_{12} + \min(0, w) \geq \hat \psi$ holds a.e.\ in $Q$. 
This allows us to use these functions as test functions in the 
EVIs satisfied by $y_1$, $y_2$, and $y_{12}$
and to obtain that 
\begin{equation*}
\begin{aligned}
\alpha \int_0^T \left \langle \partial_t y_1  -  \Delta y_1 -  u_1, - \min(0, w)\right  \rangle_{H_0^1(\Omega)}\mathrm{d}t \geq 0,
\\
(1-\alpha)
\int_0^T \left \langle  \partial_t y_2  -   
\Delta y_2 -  u_2, - \min(0, w)\right  \rangle_{H_0^1(\Omega)}\mathrm{d}t \geq 0,
\\
\int_0^T \left \langle \partial_t y_{12}  -  \Delta y_{12}  - \alpha u_1 - (1-\alpha)u_2, \min(0, w)\right  \rangle_{H_0^1(\Omega)}\mathrm{d}t \geq 0.
\end{aligned}
\end{equation*}
Adding the above and 
invoking \cite[Lemma A.1]{ChristofVexler2020} 
provides us with 
\begin{equation*}
\begin{aligned}
	0 &\geq  
    \int_0^T \left \langle \partial_t w  -  \Delta w, \min(0, w)\right  \rangle_{H_0^1(\Omega)}\mathrm{d}t 
    \geq
    \| \nabla \min(0, w) \|_{L^2([L^2]^d)}^2. 
\end{aligned}
\end{equation*}
Due to Poincaré's inequality, the above yields 
$\min(0, w) = 0$ a.e.\ in $Q$ and, thus, 
\begin{equation}
\label{eq:pointwise_convex_S}
    S_{\hat \psi}(\alpha u_1 + (1 - \alpha)u_2)
    \leq
    \alpha S_{\hat \psi}(u_1)
    +
    (1-\alpha)
    S_{\hat \psi}(u_2)\quad \text{a.e.\ in } Q
\end{equation}
as desired. 
Note that, since 
$y_1, y_2, y_{12} \in C(\bar Q)$ holds 
by \cref{lem_deg_obstacle_probs}, 
\eqref{eq:pointwise_convex_S} is even true everywhere 
in $\bar Q$ if continuous representatives are considered.
By continuity, we can extend \eqref{eq:pointwise_convex_S}
to the case $u_1, u_2 \in L^1(L^1)$. 
Along these lines, we obtain that 
\begin{equation}
\label{eq:pointwise_convex_S2}
    S_{\hat \psi}(\alpha u_1 + (1 - \alpha)u_2)(s)
    \leq
    \alpha S_{\hat \psi}(u_1)(s)
    +
    (1-\alpha)
    S_{\hat \psi}(u_2)(s)\quad \text{a.e.\ in } \Omega
\end{equation}
holds for all $u_1, u_2 \in L^1(L^1)$, $\alpha \in [0,1]$,
and $s \in [0,T]$. 
Assume now that $s \in [0,T]$
is fixed and that $u, g \in L^1(L^1)$ 
and a non-increasing
sequence $\{\tau_n\} \subset \left (0, \infty \right )$
with $\tau_n \to 0$ are given.
Then the  a.e.-convexity 
in \eqref{eq:pointwise_convex_S2}
implies that the difference quotients 
\[
\delta_n^g := \frac{S_{\hat \psi}(u + \tau_n g)(s) - S_{\hat \psi}(u)(s)}{\tau_n} \in L^1(\Omega),\qquad n \in \N,
\]
are non-increasing a.e.\ in $\Omega$ for $n \to \infty$.
From the lemma of Fatou and 
the 1-Lipschitz continuity of 
$S_{\hat \psi}\colon L^1(L^1) \to C(L^1)$, 
we further obtain that 
\[
\|g\|_{L^1(L^1)}
\geq \liminf_{n \to \infty}
\int_\Omega |\delta_n^g| \dd x
\geq
\int_\Omega \liminf_{n \to \infty}
 |\delta_n^g| \dd x. 
\]
Combined, the above two facts imply that there exists 
$\delta \in L^1(\Omega)$ such that 
\[
\lim_{n \to \infty} \delta_n^g( x) = \delta( x)
\text{ f.a.a.\ } x \in \Omega
\qquad
\text{and}
\qquad 
L^1(\Omega) \ni \delta \leq \delta_n^g \leq \delta_1^g \in L^1(\Omega)
\text{ a.e.\ in } \Omega. 
\]
From the dominated convergence theorem, we now get 
$\delta_n^g \to \delta$ in $L^1(\Omega)$ for \mbox{$n \to \infty$}. 
Using standard contradiction arguments (see
\cite[proof of Lemma 4.6]{Christof2019}), 
one easily checks that the limit 
$\delta$ is the same for all choices of $\{\tau_n\}$
and that $\delta_n^g \to \delta$ in $L^1(\Omega)$
also holds for arbitrary 
$\{\tau_n\} \subset \left (0, \infty \right )$
with $\tau_n \to 0$.
This shows that
$S_{\hat \psi}$ satisfies the directional 
differentiability condition in \eqref{eq:had_probs_unilatera}
with  
\smash{$S_{\hat \psi}'(u;g)(s) := \delta$}
for all $\{g_n\}$ with  
$g_n = g$ for all $n$.
That \eqref{eq:had_probs_unilatera} 
also holds for arbitrary sequences 
$\{g_n\}$ follows 
from the Lipschitz continuity 
of $S_{\hat \psi}\colon L^1(L^1) \to C(L^1)$ and \cite[Proposition~2.49]{BonnansShapiro2000}. 

It remains to prove the Hadamard directional 
differentiability of the map \smash{$M_{\hat \psi}$}. 
To this end, we assume that $u, g \in L^1(L^1)$,
$\{g_n\} \subset L^1(L^1)$,
and
$\{\tau_n\} \subset \left (0, \infty \right )$
satisfying 
$g_n \to g$ in $L^1(L^1)$
and $\tau_n \to 0$ are given
and define
\[
\mu_n^{g_n} := \frac{M_{\hat \psi}(u + \tau_n g_n) - M_{\hat \psi}(u)}{\tau_n} \in L^1(L^1).
\]
From \eqref{eq:M_Lip_1}, we get that 
$\{\mu_n^{g_n}\}$ is bounded in $L^1(Q)$ and, 
consequently, also bounded in $\MM((0,T] \times \Omega)$. 
Due to the theorem of Banach--Alaoglu, 
this implies that every subsequence of 
$\{\mu_n^{g_n}\}$ has a subsequence that 
converges weakly-star in $\MM((0,T] \times \Omega)$. 
From the definitions of $M_{\hat \psi}$
and $\mu_n^{g_n}$ and integration by 
parts,
we further get that 
\begin{equation*}
\begin{aligned}
\left \langle \mu_n^{g_n}, v\right \rangle_{C_0((0,T] \times \Omega)}
&=
\int_Q g_n v + \left (\partial_t v + \Delta v \right )
\frac{S_{\hat \psi}(u + \tau_n g_n) - S_{\hat \psi}(u)}{\tau_n}
\dd (t,x)
\\
&-
\int_\Omega v(T) 
\frac{S_{\hat \psi}(u + \tau_n g_n)(T) - S_{\hat \psi}(u)(T)}{\tau_n}
\dd x
~~\forall v \in C_c^\infty((0,T] \times \Omega).
\end{aligned}
\end{equation*}
In combination with the dominated convergence theorem,
the already established Hadamard 
directional differentiability of 
$S_{\hat \psi}$, and the density 
of $C_c^\infty((0,T] \times \Omega)$
in $ C_0((0,T] \times \Omega)$, 
the above implies that 
all weak-star accumulation points of $\{\mu_n^{g_n}\}$
for $n \to \infty$ are the same and 
that the resulting weak-star limit 
only depends on the limit $g$ of the 
sequence $\{g_n \}$.
This establishes the 
weak-star Hadamard directional differentiability 
of $M_{\hat \psi}$ in \eqref{eq:had_probs_unilatera}
and completes the proof. 
\end{proof}

We remark that 
\cref{lem:dir_diff_uni}
holds analogously 
for the unilateral case with an upper obstacle, i.e., 
$V_\psi = \emptyset$ and 
$\chi_\psi \equiv -\infty$.
We omit stating this as a separate result here. 
Note further that 
\cref{lem:dir_diff_uni}
also immediately implies the 
extendability 
and directional differentiability
of the maps 
$M_\psi$ and $M^\phi$.

\section{Directional derivatives for the  bilateral parabolic obstacle problem}
\label{sec:5}

We now have all preliminaries in place
to prove the directional differentiability 
of the solution map $S_\psi^\phi$
of the bilateral parabolic obstacle problem \eqref{eq:bilateral_obst_intro}. 
Note that, for this map, 
pointwise-a.e.\ curvature properties à la 
\eqref{eq:pointwise_convex_S} are not available.
One thus indeed needs new ideas here.
Throughout this section, we assume that the 
following holds.

\begin{assumption}[standing assumptions for \cref{sec:5}]~
\label{ass:standing:obstacle}~
\begin{enumerate}[label=\roman*)]
\item $T>0$ is a given terminal time;
\item $\Omega \subset \R^d$, $d \in \{1,2,3\}$, 
is a bounded domain that is 
convex or of class $C^{1,1}$;
\item $Q := (0,T) \times \Omega$ denotes the space-time cylinder;
\item $y_0 \in H_0^1(\Omega) \cap H^2(\Omega)$ 
is a given and fixed initial datum;
\item\label{ass:standing:obstacle:v} $\psi, \phi \in L^\infty(H^2) \cap W^{1,\infty}(L^2)$ 
are such that there exists $\varepsilon > 0$
with:
\begin{enumerate}[label=\alph*)]
\item $\psi(0) + \varepsilon 
\leq y_0 \leq \phi(0) - \varepsilon$ in $\bar \Omega$;
\item  
$ \psi \leq - \varepsilon \leq \varepsilon \leq \phi$
on $[0,T] \times \partial \Omega$;
\item 
$\psi + \varepsilon \leq \phi$ in $\bar Q$;
\end{enumerate}
\item $\bar u \in L^\infty(L^2)$ is given and fixed.
\end{enumerate}
\end{assumption}
\pagebreak

Note that 
$y_0$, $\psi$,
and $\phi$ possess 
$C(\bar \Omega)$-,
respectively, 
$C(\bar Q)$-representatives 
in the situation of
Assumption
\ref{ass:standing:obstacle}
since 
$H^2(\Omega) \hookrightarrow C(\bar \Omega)$
and 
$ L^\infty(H^2) \cap W^{1,\infty}(L^2) \hookrightarrow C(\bar Q)$;
see 
\cite[Theorem 1.41]{Troianiello1987}
and
\cite[Theorem 1.34]{Hafemeyer2020}.
The conditions 
in 
\ref{ass:standing:obstacle:v}
can thus indeed be formulated pointwise. 
Note further that  
all of the results 
of \cref{sec:4} can be applied under
Assumption \ref{ass:standing:obstacle}.
Accordingly, 
we also keep using the notation for 
solution operators and 
multiplier maps from \cref{def:sol_ops} in this section. 
The starting point of our analysis is the following 
observation.

\begin{lemma}[fixed-point equation, unmodified version]%
\label{lem:random_lem25}%
If, for every $u \in L^\infty(L^2)$,
we define 
$\lambda^\phi(u) := \max(0,  M_{\psi}^\phi(u))$
and 
$\lambda_\psi(u) := \min(0,  M_{\psi}^\phi(u))$,
then it holds 
\begin{equation}
\label{eq:FP_bilateral}
\begin{matrix}
\lambda^\phi(u) = M^\phi(u - \lambda_\psi(u))
\\
\lambda_\psi(u) = M_\psi(u - \lambda^\phi(u))
\end{matrix}
\qquad 
\forall u \in L^\infty(L^2).
\end{equation}
Further,
for all $u \in L^\infty(L^2)$,
we have that 
$S_\psi^\phi(u)
=
S^\phi(u - \lambda_\psi(u))
=
S_\psi(u - \lambda^\phi(u))
$.
\end{lemma}

\begin{proof}
For $u \in L^\infty(L^2)$,
we have \smash{$y := S_\psi^\phi(u) \in C(\bar Q)$}
by \cref{cor:continuity_solutions}.
Since 
$\psi, \phi \in C(\bar Q)$
and 
$\psi + \varepsilon \leq \phi$ in $\bar Q$ for some $\varepsilon > 0$, it follows that the sets 
$\{y = \psi\}$ and $\{y = \phi\}$
(defined as subsets of $\bar Q$ w.r.t.\ $C(\bar Q)$-representatives)
are disjoint. This, in turn, implies  
in combination with the definition of
\smash{$M_\psi^\phi(u)$} that
the 
functions $\lambda^\phi$ and $\lambda_\psi$
in the complementarity system 
\eqref{eq:bilateral_obst_stat_sys} for $u$ 
are precisely 
\smash{$\max(0,  M_{\psi}^\phi(u))$}
and 
\smash{$\min(0,  M_{\psi}^\phi(u))$}. 
The assertions now follow immediately 
by comparing \eqref{eq:bilateral_obst_stat_sys}
with the systems 
\eqref{eq:unilateral_lower_obst_stat_sys}
and
\eqref{eq:unlateral_upper_obst_stat_sys}
for the unilateral cases. 
\end{proof}

Note that \eqref{eq:FP_bilateral}
is precisely of the form \eqref{eq.M_FP},
respectively, \eqref{eq:F}.
From \cref{lem:mult_contractive}, we also get
that the functions 
$\lambda^\phi, \lambda_\psi\colon L^\infty(L^2) \to L^1(L^1)$ 
in \cref{lem:random_lem25}
are globally Lipschitz continuous and
that a Lipschitz estimate of 
the type
\eqref{eq:Lipschitz_estimate_abstract}
holds in the situation of \eqref{eq:FP_bilateral}
in $[L^1(L^1)]^2$.
There are, however, two main issues
at this point that 
prevent the application 
of \cref{th:abstract_dir_diff} 
to \eqref{eq:FP_bilateral}.
First, in Assumption \ref{ass:standing:abstract},
we require that $F$ lives on 
a dual space $\Lambda$ that is 
compatible with the Lipschitz estimate 
satisfied by $F$ and $G$. This is
not the case 
in \eqref{eq:FP_bilateral} where, presently, 
all functions map between $L^p(L^q)$-spaces and
the Lipschitz estimates involve $L^1(L^1)$-norms. 
Second, for the multiplier maps 
$M^\phi$ and $M_\psi$ appearing on the right-hand side
of \eqref{eq:FP_bilateral}, 
we only have the Hadamard 
directional differentiability results in 
\cref{lem:dir_diff_uni} which only allow 
for strongly convergent sequences of 
directions $\{g_n\}$ but not for 
weakly-star convergent ones,
as required in 
Assumption \ref{ass:standing:abstract}\ref{ass:standing:abstract:iv}c).
In what follows, we show that the above 
two problems can be overcome by means 
of the extension results in \cref{lem:dir_diff_uni}
and a localization argument that has, 
in a different context, 
already been used in 
\cite[sect.~6]{ChristofWachsmuth2025}
to prove that the solution map of the 
elliptic bilateral obstacle problem 
is Newton differentiable.
We begin by noting that we have the following lemma.

\begin{lemma}[stability and disjointness of contact sets]%
\label{lem_disjoint_contact_sets}%
There exist $r>0$
and nonempty compact disjoint sets 
$Z_\psi, Z^\phi \subset (0, T] \times \Omega$ such that
\begin{align*}
\{S_\psi^\phi(u) = \psi \}
\subset Z_\psi
\quad \text{and}\quad
\{S_\psi^\phi(u) = \phi \}
\subset Z^\phi
\qquad 
\forall u \in B_r^{L^\infty(L^2)}(\bar u). 
\end{align*}
Here, the contact sets are defined
as subsets of $\bar Q$ w.r.t.\
the $C(\bar Q)$-representatives of the involved functions. 
\end{lemma}

\begin{proof}
From \cref{lem:Lipschitz}, we know that 
$S_\psi^\phi$ is globally Lipschitz continuous 
as a function from $L^\infty(L^2)$ to $C(\bar Q)$.
Because of this continuity
and our assumptions on $\psi$ and $\phi$,
we can argue along the exact same lines 
as in the proof of 
\cite[Lemma 6.14]{ChristofWachsmuth2025} to obtain the assertion of the lemma. 
\end{proof}

\Cref{lem_disjoint_contact_sets} allows us 
to modify the 
right-hand side 
of \eqref{eq:FP_bilateral} 
such that 
it involves obstacles only on 
those parts of 
$\bar Q$ that are relevant 
for parameters $u$ near 
$\bar u$.
To formalize this idea, 
we introduce the following setting.

\begin{assumption}[{setting for the local 
modification of \eqref{eq:FP_bilateral}}]%
\label{ass:localization_setting}
\begin{enumerate}[label=\roman*)]
\item 
$r>0$, $Z_\psi$, and $Z^\phi$ are
as in \cref{lem_disjoint_contact_sets};
\item 
$V_\psi, V^\phi \subset (0, \infty) \times \Omega$ 
are open 
bounded sets 
satisfying 
$Z_\psi \subset V_\psi$,
$Z^\phi \subset V^\phi$, 
$\bar V_\psi, \bar V^\phi \subset (0, \infty) \times \Omega$,
and 
$\dist(V_\psi, V^\phi) > 0$;
\item $W_\psi, W^{\phi} \subset 
(0, \infty) \times \Omega$ are open bounded  sets
satisfying
$\bar V_\psi \subset W_\psi$,
$\bar V^\phi\subset W^\phi$,
$\bar W_\psi, \bar W^\phi \subset (0, \infty) \times \Omega$,
and 
$\dist(W_\psi, W^\phi) > 0$;
\item 
$\chi_\psi\colon \R^{d+1} \to [-\infty, 0]$
is an upper semicontinuous function that is
real-valued on $V_\psi$ and satisfies
$\chi_\psi|_{V_\psi} \in C^2(V_\psi)$,
$\chi_\psi \equiv 0$ on $Z_\psi$,
$\chi_\psi \equiv - \infty$ on 
$\R^{d+1} \setminus V_\psi$;
\item 
$\chi^\phi\colon \R^{d+1} \to [0, \infty]$ 
is a lower semicontinuous function 
that is real-valued on $V^\phi$ and satisfies 
$\chi^\phi|_{V^\phi} \in C^2(V^\phi)$,
$\chi^\phi \equiv 0$ on $Z^\phi$,
$\chi^\phi \equiv \infty$ on 
$\R^{d+1} \setminus V^\phi$;
\item 
$\zeta_\psi 
\in C^\infty(\R^{d+1}) $ is a function satisfying 
$0 \leq \zeta_\psi \leq 1$ in $\R^{d+1}$,
$\zeta_\psi \equiv 0$ in an open neighborhood of 
$\bar V_\psi$, $\zeta_\psi \equiv 1$ in 
$\R^{d+1} \setminus W_\psi$;
\item
$\zeta^\phi \in C^\infty(\R^{d+1}) $ is a 
function satisfying 
$0 \leq \zeta^\phi \leq 1$ in $\R^{d+1}$,
$\zeta^\phi \equiv 0$ 
in an open neighborhood of $\bar V^\phi$,
$\zeta^\phi \equiv 1$ in $\R^{d+1} \setminus W^\phi$;
\item $\hat \psi := \psi + \smash{\chi_\psi}$ 
and $\hat \phi := \phi + \chi^\phi$
are as in \eqref{eq:hat_obstacles},
and 
$\smash{M_{\hat \psi}, M^{\hat \phi}}\colon L^1(L^1) \to L^1(L^1)$
and 
$S_{\hat \psi}, S^{\hat \phi}\colon L^1(L^1) \to C(L^1)$
are
the (extended) multiplier 
and solution maps of the unilateral 
parabolic obstacle problems 
that have only the lower obstacle $\hat \psi$,
respectively, only the upper obstacle $\hat \phi$;
see \cref{lem_deg_obstacle_probs,lem:dir_diff_uni}.
\end{enumerate}
\end{assumption}

Note that, due to the nonemptiness, compactness, 
and disjointness of 
the sets $Z_\psi$ and $Z^\phi$ in \cref{lem_disjoint_contact_sets},
sets $V_\psi, V^\phi, W_\psi, W^\phi$ and 
functions $\chi_\psi, \chi^\phi, \zeta_\psi, \zeta^\phi$
with the properties in 
Assumption \ref{ass:localization_setting}
can easily be constructed. 
As a first modification of \eqref{eq:FP_bilateral},
we now obtain:

\begin{lemma}[{first modification of \eqref{eq:FP_bilateral}}]%
\label{lem:first_mod}%
Consider the situation in 
Assumption 
\ref{ass:localization_setting} and 
define 
$\lambda^\phi(u) := \max(0,  M_{\psi}^\phi(u))$
and 
$\lambda_\psi(u) := \min(0,  M_{\psi}^\phi(u))$
for all $u \in L^\infty(L^2)$. 
Then it holds 
\begin{equation}
\label{eq:FP_bilateral_mod1}
\begin{matrix}
\lambda^\phi(u) = M^{\hat\phi}(u - \lambda_\psi(u))
\\
\lambda_\psi(u) = M_{\hat \psi}(u - \lambda^\phi(u))
\end{matrix}
\qquad 
\forall u \in B_r^{L^\infty(L^2)}(\bar u).
\end{equation}
Further,
for all $u \in B_r^{L^\infty(L^2)}(\bar u)$,
we have 
$S_\psi^\phi(u)
=
S^{\hat \phi}(u - \lambda_\psi(u))
=
S_{\hat \psi}(u - \lambda^\phi(u))
$.
\end{lemma}

\begin{proof}
If $u \in B_r^{L^\infty(L^2)}(\bar u)$ with 
associated $y := S_\psi^\phi(u)$ is given, 
then we know from \cref{lem_disjoint_contact_sets}
that $\{y = \psi\} \subset Z_\psi$ and 
$\{y = \phi\} \subset Z^\phi$. In view of 
the system 
\eqref{eq:bilateral_obst_stat_sys}
and \cref{cor:continuity_solutions},
this implies that 
$y$, $u$, 
$\lambda_\psi(u)$,
and
$\lambda^\phi(u)$
satisfy 
\begin{equation*}
\begin{gathered}
y \in L^2(H_0^1\cap H^2) \cap H^1(L^2),
\quad 
\lambda_\psi(u), \lambda^\phi(u) \in L^\infty(L^2),
\quad y(0) = y_0, 
\\
\partial_t y  - \Delta y  = 
u - \lambda_\psi(u) - \lambda^\phi(u)\quad \text{ a.e.\ in }Q,
\\
\lambda_\psi(u) \leq 0,
\qquad  
y - \psi \geq 0,
\qquad
\lambda_\psi(u)(y - \psi) = 0\quad \text{ a.e.\ in }Q,
\\
\lambda^\phi(u) \geq 0,
\qquad  
y - \phi \leq 0,
\qquad
\lambda^\phi(u)(y - \phi) = 0\quad \text{ a.e.\ in }Q,
\\
\lambda_{\psi}(u) = 0~\text{ a.e.\ in }Q \setminus Z_\psi,
\qquad
\lambda^{\phi}(u) = 0~\text{ a.e.\ in }Q \setminus Z^\phi.
\end{gathered}
\end{equation*}
If we compare the above with the 
complementarity systems
for the unilateral obstacle problems 
that involve only
the obstacle $\hat \psi$, respectively, 
$\hat \phi$, see \eqref{eq:bilateral_obst_stat_sys_deg},
keeping in mind that $\phi = \hat  \phi$
a.e.\ on $Z^\phi$, 
$\psi = \hat  \psi$
a.e.\ on $Z_\psi$,
$\hat \psi \leq \psi$ 
a.e.\ on $Q$,
and 
$\hat \phi \geq \phi$ a.e.\ on $Q$,
then the assertions 
of the lemma follow immediately
from \cref{lem_deg_obstacle_probs}.
\end{proof}

Note that the fixed-point equation 
\eqref{eq:FP_bilateral_mod1} still does not fit 
into the framework of \cref{sec:3}. 
To modify it further, we 
have to extend it to spaces of measures and 
insert an additional operator that creates 
compactness. We prepare this by proving:

\begin{lemma}[strict density in spaces of measures]%
\label{lemma:strict_density}%
Let $V \subset \R^N$, $N \in \N$, 
be a nonempty, open, bounded set.
Then there exist maps 
$R_{m,n} \in \LL(\MM(\bar V), C(\bar V))$
and 
$R_{n} \in \LL(\MM(\bar V), \MM(\bar V))$, $m,n \in \N$, such that,
for all $\mu \in \MM(\bar V)$, it holds 
\begin{gather*}
R_{m,n}\mu \in C_c^\infty(V)~\forall m,n \in \N,
\\
R_{m,n}\mu \weaklystar R_n\mu \text{ in } \MM(\bar V) 
~\text{and}~
\|R_{m,n}\mu\|_{\MM(\bar V)} 
\to 
\|R_{n}\mu\|_{\MM(\bar V)} \text{ for } m \to \infty, \forall n\in \N,
\\
R_{n}\mu \weaklystar \mu \text{ in } \MM(\bar V) 
~\text{and}~
\|R_{n}\mu\|_{\MM(\bar V)} 
\to 
\|\mu\|_{\MM(\bar V)} \text{ for } n \to \infty.
\end{gather*}
\end{lemma}

\begin{proof}
We consider for every $n \in \N$ the partition of 
$\R^N$ into the disjoint cells 
$ Q_{i,n} := (1/n)(i + [0,1)^N )$, $i \in \Z^N$,  define 
$\II_n := \{i \in \Z^N \mid \bar V \cap Q_{i,n} \neq \emptyset\}$,
and choose for every $i \in \II_n$ an arbitrary 
but fixed $x_{i,n} \in \bar V \cap Q_{i,n}$. For this selection, we define 
\begin{equation}
\label{eq:RnDef}
    R_{n}\mu := \sum_{i \in \II_n} \mu(\bar V \cap Q_{i,n})\delta_{x_{i,n}}
    \in \MM(\bar V)\qquad \forall \mu \in \MM(\bar V),
\end{equation}
where $\delta_{x_{i,n}}$ denotes a Dirac measure supported at $x_{i,n}$. 
Note that \eqref{eq:RnDef} trivially defines a map $R_n \in \LL(\MM(\bar V), \MM(\bar V))$. For every $\mu \in \MM(\bar V)$ and $w \in C(\bar V)$, we further have 
\begin{align}
\label{eq:randomeq272g378n}
\left \langle R_n \mu, w\right \rangle_{C(\bar V)}
=
\sum_{i \in \II_n} \mu(\bar V \cap Q_{i,n})w(x_{i,n})
=
\int_{\bar V}
\sum_{i \in \II_n} \mathds{1}_{\bar V \cap Q_{i,n}} w(x_{i,n})
\dd \mu.
\end{align}
Since $\bar V$ is compact, $\diam(Q_{i,n}) = \sqrt{N}/n$,
and continuous functions on compact sets are 
uniformly continuous, 
we further have for all $w \in C(\bar V)$ that 
\[
\sup_{x \in \bar V} \left | w(x) - \sum_{i \in \II_n} \mathds{1}_{\bar V \cap Q_{i,n}}(x) w(x_{i,n}) \right | 
=
\sup_{i \in \II_n}
\left (
\sup_{x \in \bar V \cap Q_{i,n}}  | w(x) - w(x_{i,n})   |
\right )
\to 0
\]
for $n \to \infty$. Due to
\eqref{eq:randomeq272g378n}
and the dominated convergence theorem, this yields
\[
\left \langle R_n \mu, w\right \rangle_{C(\bar V)} 
\to 
\left \langle \mu, w\right \rangle_{C(\bar V)}
\qquad \forall w \in C(\bar V)
\qquad \forall \mu \in \MM(\bar V),
\]
i.e., we have $R_n\mu \weaklystar \mu$ in $\MM(\bar V)$
for $n \to \infty$ for all $\mu \in \MM(\bar V)$. 
Note that \eqref{eq:randomeq272g378n},
the identity 
$\|\mu\|_{C(\bar V)^*} = \|\mu\|_{\MM(\bar V)} = |\mu|(\bar V)$
obtained from the theorem of Riesz--Alexandroff
\cite[Theorem 2.4.6]{Attouch2006},
and the definition of the total variation measure $|\mu|(\bar V)$ \cite[Definition 1.4]{Ambrosio2000}
further imply
\begin{align*}
    \|R_n \mu\|_{\MM(\bar V)}
    =
    \sup_{w \in B_1^{C(\bar V)}(0)}
    \left \langle R_n \mu, w\right \rangle_{C(\bar V)}
    \leq
    \sum_{i \in \II_n}   |\mu(\bar V \cap Q_{i,n})  |
    \leq
    |\mu|(\bar V)
     =
    \|\mu\|_{\MM(\bar V)}
\end{align*}
for all $n \in \N$ and $\mu \in \MM(\bar V)$.
As the norm 
$\|\cdot\|_{\MM(\bar V)}$ is weak-star lower semicontinuous
and
\smash{$R_n\mu \weaklystar \mu$} for all 
$\mu \in \MM(\bar V)$, this shows that 
$\|R_n \mu\|_{\MM(\bar V)} \to \|\mu\|_{\MM(\bar V)}$
holds for $n \to \infty$ for all $\mu \in \MM(\bar V)$.
The operators $R_{n} \in \LL(\MM(\bar V), \MM(\bar V))$, $n \in \N$,
thus have all of the desired properties. 
It remains to construct the functions $R_{m,n}$.
To this end, we fix $n \in \N$ and consider for 
the points $x_{i,n} \in \bar V$, $i \in \II_n$,
closed Euclidean 
balls $B_{\varepsilon_m}(x_{i,n,m})$, $m \in \N$,
with centers $x_{i,n,m} \in V$ and radii $\varepsilon_m > 0$
such that:
\begin{enumerate}[label=\roman*)]
\item $\varepsilon_m \to 0$ 
and $x_{i,n,m} \to x_{i,n}$ for all $i \in \II_n$ for $m \to \infty$;
\item the balls $B_{\varepsilon_m}(x_{i,n,m})$
are subsets of $V$ for all $i \in \II_n$, $m \in \N$;
\item for all $m$, the balls 
$\{B_{\varepsilon_m}(x_{i,n,m})\}_{i \in \II_n}$
have positive distance from each other. 
\end{enumerate} 
\pagebreak

It is easy to check that the fact that $\bar V$
is the closure of an open set implies that 
$\varepsilon_m$ and $x_{i,n,m}$ with the above properties can always be found. We now define
\begin{equation}
\label{eq:Rmn_def}
    R_{m,n}\mu := \sum_{i \in \II_n} \mu(\bar V \cap Q_{i,n})  
    \frac{1}{\varepsilon_m^N}
    \varrho \left(\frac{\cdot - x_{i,n,m}}{\varepsilon_m}\right)
    \in C_c^\infty(V)\qquad \forall \mu \in \MM(\bar V)
\end{equation}
 for all $m \in \N$,
where $ \varrho \in C_c^\infty(\R^N)$
denotes a standard mollifier; see
\cite[sect.\ 2.2.2]{Attouch2006}.
From \eqref{eq:Rmn_def}, it follows immediately that
$R_{m,n} \in \LL(\MM(\bar V), C(\bar V))$ for all $m$.
Using the mean value theorem and the fact that 
$\|z\|_{L^1(V)} = \|z\|_{\MM(\bar V)}$ for all $z \in L^1(V)$, 
one further easily checks that 
\[
\left \langle R_{m,n} \mu, w\right \rangle_{C(\bar V)} 
\to 
\left \langle R_n \mu, w\right \rangle_{C(\bar V)}
\quad \forall w \in C(\bar V)
\quad \forall \mu \in \MM(\bar V)\quad \forall n \in \N
\quad \text{ as }m \to \infty
\]
and that 
\[
\|R_{m,n}\mu\|_{\MM(\bar V)} 
=
\|R_{m,n}\mu\|_{L^1(V)} 
=
\sum_{i \in \II_n} |\mu(\bar V \cap Q_{i,n})    |
=
\|R_{n}\mu\|_{\MM(\bar V)}
\]
holds for all $m,n \in \N$ and all $\mu \in \MM(\bar V)$.
This shows that the operators $R_{m,n}$ 
have all of the desired properties 
and completes the proof of the lemma. 
\end{proof}

\begin{remark}
Using a diagonal sequence,
the theorem of Banach--Steinhaus,
and 
the separability of $C(\bar V)$,
one easily checks that 
\cref{lemma:strict_density}
implies the well-known fact that 
$C_c^\infty(V)$ is strictly dense in $\MM(\bar V)$
for open, bounded, nonempty $V \subset \R^N$.
\end{remark}

To extend \eqref{eq:FP_bilateral_mod1}
to measure spaces, we also need 
results on the heat equation with
right-hand sides in $\MM(\bar Q)$.
We therefore introduce:

\begin{definition}[notation for the heat equation]
We denote by:
\begin{enumerate}[label=\roman*)]
\item
$D\colon L^1(W^{2,1}) \cap W^{1,1}(L^1) \to L^1(L^1)$
the differential 
operator 
$Dv := \partial_t v - \Delta v$;
\item $D^{-1}\colon L^2(L^2) \to L^2(H_0^1 \cap H^2)
\cap H^1(L^2)$ the solution map 
$\mu \mapsto w$ of 
\[
\partial_t w - \Delta w = \mu \text{ in } Q,
\quad
w = 0 \text{ on } (0,T) \times \partial\Omega,
\quad
w = 0 \text{ in } \{0\} \times \Omega.
\]
\end{enumerate}
\end{definition}

Note that $D^{-1}$ is 
well defined as a function 
from $L^2(L^2)$ to $L^2(H_0^1 \cap H^2)
\cap H^1(L^2)$ 
by \cite[Proposition 2.3]{ChristofVexler2020},
that $DD^{-1}\mu = \mu$ holds for all 
$\mu \in L^2(L^2)$, and that $D^{-1}$
can be extended to larger domains of definition 
by continuity (e.g., to 
a map 
$D^{-1}\colon L^1(L^1) \to C(L^1)$ by
\cref{lem:dir_diff_uni}). 
These extensions to larger spaces 
are also denoted
by $D^{-1}$
for simplicity. 
For our analysis of \eqref{eq:FP_bilateral_mod1},
we require in particular:

\begin{lemma}[extension of $D^{-1}$
to measure spaces]%
\label{lem:measure_heat}%
Let $V \subset Q$
be nonempty and open,
identify  
$L^2(V)$ with a subset of $L^2(Q) \cong L^2(L^2)$ 
via extension by zero,
consider $D^{-1}$ as a map
$D^{-1}\colon L^2(V) \to L^2(H_0^1\cap H^2) \cap H^1(L^2)$,
and fix $1 < p < 5/4$.
Then $D^{-1}$
admits a unique weak-star-to-weak 
continuous extension 
$D^{-1}\colon \MM(\bar V) \to L^p(W^{1,p})$
and this extension satisfies 
$D^{-1} \in \LL(\MM(\bar V), L^p(W^{1,p}))$
and%
\begin{equation}
\label{eq:extension_characterization}
w = D^{-1}\mu
~~\Leftrightarrow~~
w \in L^1(Q),
~
\int_Q 
w
\left (
-\partial_t z - \Delta z
\right )
\dd(t,x)
=
\left \langle \mu , z\right \rangle_{C(\bar V)}
~\forall z \in \ZZ
\end{equation}
with
$
\ZZ:=
\left \{
z \in L^2(H_0^1 \cap H^2) \cap H^1(L^2)
\mid
\partial_t z + \Delta z \in L^\infty(Q)\text{ and }z(T) = 0 
\right \} \subset C(\bar Q).
$
Further, for every 
$\zeta \in C^\infty(\R^{d+1})$ satisfying 
$\zeta \equiv 0$ in an open neighborhood of 
$\bar V$, one has that the function 
\begin{equation}
\label{eq:D_props}
    \MM(\bar V) \ni \mu \mapsto \zeta D^{-1}\mu \in 
     L^p(W^{1,p}_0 \cap W^{2,p}) \cap W^{1,p}(L^p)
\end{equation}
is well defined, linear, and weak-star-to-strong continuous. 
\end{lemma}
\pagebreak

\begin{proof}
That $D^{-1}$ can have at most one 
weak-star-to-weak continuous extension 
$D^{-1}\colon \MM(\bar V) \to L^p(W^{1,p})$
follows from the weak-star density of 
$C_c^\infty(V)$ in $\MM(\bar V)$.
To prove that such an extension exists,
is in $\LL(\MM(\bar V), L^p(W^{1,p}))$,
and is characterized by  
\eqref{eq:extension_characterization}, 
one can follow word for word
the proof of \cite[Theorem 2.2]{Casas2016}.
(Note that, for the strict approximation argument
in \cite{Casas2016}, one can also use our 
\cref{lemma:strict_density}, and 
that $\ZZ \subset C(\bar Q)$ follows
from \cref{lem_deg_obstacle_probs}
with $\chi_\psi \equiv -\infty$ and $\chi^\phi \equiv \infty$.)
It remains to prove the assertions on the map 
\eqref{eq:D_props}. So let us assume that 
$\zeta \in C^\infty(\R^{d+1})$ 
is an arbitrary but fixed function satisfying 
$\zeta \equiv 0$ in an open neighborhood of 
$\bar V$. 
From \cite[Theorem~6.2]{Wood2007}, we get that 
$D^{-1}$ is also
well defined, linear, and continuous as a map
$D^{-1}\colon L^{p}(L^{p}) \to 
L^p(W_0^{1,p}\cap W^{2, p}) \cap W^{1,p}(L^p)$.
(Note that, if $\Omega$ is of class $C^{1,1}$,
then the outer ball condition
needed in \cite[Theorem 6.2]{Wood2007}
follows from \cite[Corollary~2]{Lewicka2020}.)
Using the product rule, the 
weak-star-to-weak continuity of $D^{-1}$,
the weak-star density of 
$C_c^\infty(V)$ in $\MM(\bar V)$,
and the properties
of $\zeta$, one further easily checks that,
for every $\mu \in \MM(\bar V)$,
one has $(\partial_t\zeta - \Delta \zeta)D^{-1}\mu - 2 \nabla \zeta \cdot \nabla D^{-1}\mu \in L^p(L^p)$ and 
\begin{equation}
\label{eq:reg_id}
 \zeta D^{-1}\mu = D^{-1}\left [(\partial_t\zeta - \Delta \zeta)D^{-1}\mu - 2 \nabla \zeta \cdot \nabla D^{-1}\mu \right ].
\end{equation}
Since 
$D^{-1}$ is well defined 
as an element of 
$\LL(\MM(\bar V), L^p(W^{1,p}))$
and
as an element of 
$\LL( L^{p}(L^{p}),
L^p(W_0^{1,p}\cap W^{2, p}) \cap W^{1,p}(L^p))$,
\eqref{eq:reg_id} already implies that 
\eqref{eq:D_props} is well defined, linear, and continuous.
It remains to prove the weak-star-to-strong continuity
of \eqref{eq:D_props}.
Let
$\{\mu_n \} \subset \MM(\bar V)$ 
be a sequence satisfying 
\smash{$\mu_n \weaklystar \mu$ in $\MM(\bar V)$}. 
If $Q \subset\bar V$, then \eqref{eq:D_props} is 
identical zero 
and there is nothing to prove,
so it suffices to consider the case
$Q\setminus \bar V \neq \emptyset$ in the following. 
Assume that $O \subset Q$  
is an arbitrary nonempty open set satisfying 
$\bar O \cap \bar V = \emptyset$
and, for the moment, that 
$\zeta \in C^\infty(\R^{d+1})$ 
not only satisfies 
$\zeta \equiv 0$ in an open neighborhood of $\bar V$
but also $\zeta \equiv 1$ on $\bar O$. 
Then the boundedness of 
$\{\mu_n\}$ in $\MM(\bar V)$,
the weak-star-to-weak continuity of the map
$D^{-1}\colon \MM(\bar V) \to L^p(W^{1,p})$,
and the linearity and continuity 
of  \eqref{eq:D_props} imply 
$\zeta D^{-1}\mu_n \weakly \zeta D^{-1}\mu$
in 
$L^p(W_0^{1,p}\cap W^{2, p})\cap W^{1,p}(L^p)$.
Due to the 
lemma of Aubin--Lions \cite[Lemma 7.7]{Roubicek2013},
the latter yields $\zeta D^{-1}\mu_n \to \zeta D^{-1}\mu$
strongly in $L^p(W_0^{1,p})$. 
Since $\zeta \equiv 1$ on $\bar O$ and 
as $O$ was arbitrary, this 
allows us to deduce that 
$(D^{-1}\mu_n)|_{O} \to (D^{-1}\mu)|_{O}$
in $L^p(O)$
and
$\nabla (D^{-1}\mu_n)|_{O} \to \nabla (D^{-1}\mu)|_{O}$
in $L^p(O)^{d}$
for all nonempty open $O \subset Q$ with 
$\bar O \cap \bar V = \emptyset$. 
Let us now again suppose that 
$\zeta \in C^\infty(\R^{d+1})$ 
is arbitrary with 
$\zeta \equiv 0$ in an open neighborhood of 
$\bar V$. Then  we know that 
$\partial_t \zeta$, $\nabla \zeta$, and $\Delta \zeta$
also vanish in this neighborhood
and we readily obtain from the 
strong $L^p$-convergence of 
$D^{-1}\mu_n$ 
and
$\nabla D^{-1}\mu_n$
away from $\bar V$ 
and $D^{-1} \in \LL(L^{p}(L^{p}),
 L^p(W_0^{1,p}\cap W^{2, p})
\cap W^{1,p}(L^p))$
that
\begin{equation*}
\begin{aligned}
 \zeta D^{-1}\mu_n 
 &= 
 D^{-1}\left [(\partial_t\zeta - \Delta \zeta)D^{-1}\mu_n - 2 \nabla \zeta \cdot \nabla D^{-1}\mu_n \right ]
 \\
 &\to 
 D^{-1}\left [(\partial_t\zeta - \Delta \zeta)D^{-1}\mu - 2 \nabla \zeta \cdot \nabla D^{-1}\mu \right ]
 = \zeta D^{-1}\mu
 \end{aligned}
\end{equation*}
in 
$ L^p(W_0^{1,p}\cap W^{2, p}) \cap W^{1,p}(L^p)$
for $n \to \infty$; see \eqref{eq:reg_id}.
This proves the 
weak-star-to-strong continuity 
of the function \eqref{eq:D_props}
and completes the proof of the lemma. 
\end{proof}

Using the operators $D$ and $D^{-1}$, 
we can prove the following 
lemma. 
 
\begin{lemma}[second modification of solution and multiplier maps]
\label{lem:mult_map_D}%
Consider the situation in 
Assumption 
\ref{ass:localization_setting}. 
Then, for all 
$u, \lambda^\phi,\lambda_\psi \in  L^\infty(L^2)$,
we have
\begin{equation*}
\begin{aligned}
M^{\hat \phi}(u - \lambda_\psi) 
&= M^{\hat \phi}(u - D[\zeta_\psi D^{-1}\lambda_\psi])
&&\text{ a.e.\ in } Q,
\\
M_{\hat \psi}(u - \lambda^\phi)
&= M_{\hat \psi}(u - 
D[\zeta^\phi D^{-1}\lambda^\phi]) 
&&\text{ a.e.\ in } Q,
\\
S^{\hat \phi}(u - \lambda_\psi) 
&= S^{\hat \phi}(u - D[\zeta_\psi D^{-1}\lambda_\psi])
&&\text{ in } \bar Q \setminus W_\psi,
\\
S_{\hat \psi}(u - \lambda^\phi)
&= S_{\hat \psi}(u - 
D[\zeta^\phi D^{-1}\lambda^\phi]) 
&&\text{ in } \bar Q \setminus  W^\phi.
\end{aligned}
\end{equation*}
\end{lemma}

\begin{proof}
Assume that $\lambda^\phi, u \in L^\infty(L^2)$
are given. Then 
\cref{lem_deg_obstacle_probs} implies that 
$y:= S_{\hat \psi}(u - \lambda^\phi) \in L^2(H_0^1 \cap H^2) \cap H^1(L^2) \cap C(\bar Q)$ satisfies
\begin{gather*}
\partial_t y  - \Delta y  = 
u   - \lambda^\phi - M_{\hat \psi}(u - \lambda^\phi)
\text{ a.e.\ in } 
Q,
\qquad 
M_{\hat \psi}(u - \lambda^\phi) = 0 \text{ a.e.\ in } 
Q \setminus V_\psi,
\\
M_{\hat \psi}(u - \lambda^\phi) \leq 0,
\qquad  
y - \hat \psi \geq 0,
\qquad
M_{\hat \psi}(u - \lambda^\phi) (y - \hat \psi) = 0 \text{ a.e.\ on } Q \cap V_\psi.
\end{gather*}
Define
$\tilde y := y + (1-\zeta^\phi)D^{-1}\lambda^\phi \in L^2(H_0^1 \cap H^2) \cap H^1(L^2) \cap C(\bar Q)$.
Then it holds 
\begin{align*}
\partial_t \tilde y  - \Delta \tilde y  
&= 
u   - \lambda^\phi - M_{\hat \psi}(u - \lambda^\phi)
+  D[(1-\zeta^\phi)D^{-1}\lambda^\phi]
\\
&=
u   -  D[\zeta^\phi D^{-1}\lambda^\phi] - M_{\hat \psi}(u - \lambda^\phi),
\end{align*}
we have $\tilde y(0) = y(0) = y_0$,
and it follows from the properties of $\zeta^\phi$ that 
\[
	\tilde y = y \text{ in } \bar Q \setminus W^\phi 
    \supset Q \cap W_\psi \supset Q \cap V_\psi,
\]
and, thus,
\[
M_{\hat \psi}(u - \lambda^\phi) \leq 0,
\qquad  
\tilde y - \hat \psi \geq 0,
\qquad
M_{\hat \psi}(u - \lambda^\phi) (\tilde y - \hat \psi) = 0 \text{ a.e.\ on } Q \cap V_\psi.
\]
Since $D^{-1}\lambda^\phi$ is in $L^\infty(H_0^1)$ 
by \cite[Proposition 2.3]{ChristofVexler2020} and
since $\zeta^\phi \in C^\infty(\R^{d+1})$, 
we also have 
$u - D[\zeta^\phi D^{-1}\lambda^\phi] 
=u - (D\zeta^\phi)(D^{-1}\lambda^\phi ) - \zeta^\phi\lambda^\phi
+ 2 \nabla \zeta^\phi \cdot \nabla D^{-1}\lambda^\phi 
\in L^\infty(L^2)$.
Combined, all of this allows us to 
invoke \cref{lem_deg_obstacle_probs}
again 
and to obtain that 
\[
\tilde y = S_{\hat \psi}(u   -  D[\zeta^\phi D^{-1}\lambda^\phi])
\quad\text{and}\quad
M_{\hat \psi}(u - \lambda^\phi) = M_{\hat \psi}(u   -  D[\zeta^\phi D^{-1}\lambda^\phi]).
\]
The assertions 
for $S_{\hat \psi}$ and $M_{\hat \psi}$ follow immediately
from the above identities. 
For the maps 
\smash{$S^{\hat \phi}$ and 
$M^{\hat \phi}$}, one can use the exact same arguments. 
\end{proof}

We now have all tools at hand 
to apply \cref{th:abstract_dir_diff} to \eqref{eq:bilateral_obst_intro}
and to prove the directional differentiability of 
the solution operator $S_{\psi}^\phi$.

\begin{proposition}[{application of the
abstract theory to \eqref{eq:bilateral_obst_intro}}]%
\label{prop:dir_diff}%
The function \smash{$S_\psi^\phi$} is directionally 
differentiable at $\bar u$ in the sense that, 
for all $h \in L^\infty(L^2)$, 
there exists $(S_\psi^\phi)'(\bar u;h)(T) \in L^1(\Omega)$
satisfying 
\[
\frac{S_\psi^\phi(\bar u + \tau h)(T) - S_\psi^\phi(\bar u)(T)}{\tau}
\to (S_\psi^\phi)'(\bar u;h)(T)
\text{ in } L^1(\Omega)
\text{ for } 0 < \tau \to 0.
\]
\end{proposition}

\begin{proof}
Consider the situation in 
Assumption
\ref{ass:localization_setting}
and  define
\begin{gather*}
\Lambda := \MM(\overline{Q \cap V^\phi}) \times \MM(\overline{Q \cap V_\psi}),
\qquad
\|(\lambda^\phi, \lambda_\psi)\|_\Lambda
:=
\|\lambda^\phi\|_{\MM(\overline{Q \cap V^\phi})}
+
\|\lambda_\psi\|_{\MM(\overline{Q \cap V_\psi})},
\\
\Omega^\phi
:= \{x \in \Omega \mid (T,x) \not \in \bar W_\psi \}
\neq \emptyset,
\qquad
\Omega_\psi
:= \{x \in \Omega \mid (T,x) \not \in \bar W^\phi \} \neq \emptyset,
\\
Y := 
L^1(\Omega^\phi) 
\times 
L^1(\Omega_\psi),
\qquad
\|(y^\phi, y_\psi)\|_Y
:=
\|y^\phi\|_{L^1(\Omega^\phi)}
+
\|y_\psi\|_{L^1(\Omega_\psi)},
\\
U := L^\infty(L^2),
\qquad \|u\|_U := \|u\|_{L^\infty(L^2)},
\\
F\colon \Lambda \times B_r^U(\bar u)
\to \Lambda,
\qquad
F(\lambda^\phi, \lambda_\psi, u)
:=
\begin{pmatrix}
M^{\hat \phi}(u - D[\zeta_\psi D^{-1}\lambda_\psi])
\\
M_{\hat \psi}(u - D[\zeta^\phi D^{-1}\lambda^\phi])
\end{pmatrix},
\\
M\colon  B_r^U(\bar u) \to \Lambda,
\qquad
M(u):=
\begin{pmatrix}
\max(0, M_{\psi}^\phi(u))
\\
\min(0, M_{\psi}^\phi(u))
\end{pmatrix},
\\
G\colon \Lambda \times B_r^U(\bar u)
\to Y,
\qquad
G(\lambda^\phi, \lambda_\psi, u)
:=
\begin{pmatrix}
S^{\hat \phi}(u - D[\zeta_\psi D^{-1}\lambda_\psi])(T)|_{\Omega^\phi}
\\
S_{\hat \psi}(u - D[\zeta^\phi D^{-1}\lambda^\phi])(T)|_{\Omega_\psi}
\end{pmatrix}.
\end{gather*}

Let us check that the above functions and spaces 
satisfy all of the conditions in 
Assumption \ref{ass:standing:abstract}: 
Due to the Riesz--Alexandroff 
theorem \cite[Theorem~2.4.6]{Attouch2006}, 
$\Lambda$ is the topological dual  
of the separable space 
$C(\overline{Q \cap V^\phi}) \times C(\overline{Q \cap V_\psi})$
endowed with the norm 
$\|(v^\phi, v_\psi)\|_{C(\overline{Q \cap V^\phi}) \times C(\overline{Q \cap V_\psi})}
:=
\max(\|v^\phi\|_{C(\overline{Q \cap V^\phi})}, \|v_\psi\|_{C(\overline{Q \cap V_\psi})}) $.
Further, $Y$ and $U$ are trivially normed spaces. 
Thus, $\Lambda$, $Y$, and $U$ are as required. 
Next, we verify that 
$F$, $M$, and $G$ are well defined. 
From \cref{cor:continuity_solutions} and
\cref{def:sol_ops}, we obtain that 
$M$ is well defined as a mapping 
$M\colon  B_r^U(\bar u) \to [L^\infty(L^2)]^2$. 
From \cref{lem_disjoint_contact_sets},
\eqref{eq:bilateral_obst_stat_sys}, 
and the properties of $V^\phi$ and $V_\psi$,
it follows
further 
that 
$\max(0, M_{\psi}^\phi(u)) = 0$
holds a.e.\ in $Q \setminus V^\phi$
and that 
\smash{$\min(0, M_{\psi}^\phi(u)) = 0$}
holds a.e.\ in $Q \setminus V_\psi$
for all $u \in B_r^U(\bar u)$.
This allows us to interpret
$M$ as a function
$M\colon  B_r^U(\bar u) \to \Lambda$
as desired. Next, we consider $F$ and $G$. 
From \cref{lem:dir_diff_uni},
we obtain that \smash{$M^{\hat \phi}$, $M_{\hat \psi}$,
$S^{\hat \phi}$, and $S_{\hat \psi}$}
are well defined as maps 
from $L^1(L^1)$ to $L^1(L^1)$,
respectively, $L^1(L^1)$ to $C(L^1)$.
From \cref{lem:measure_heat},
the properties of $\zeta^\phi$ and $\zeta_\psi$,
and the definition of $D$, we further get
\begin{gather*}
u - D[\zeta^\phi D^{-1}\lambda^\phi],
u - D[\zeta_\psi D^{-1}\lambda_\psi]
\in L^{p}(L^p) \subset L^1(L^1)
\\
\forall u \in L^{\infty}(L^2)
\qquad
\forall \lambda^\phi \in 
\MM(\overline{Q \cap V^\phi})
\qquad
\forall \lambda_\psi \in \MM(\overline{Q \cap V_\psi})
\qquad \forall 1 < p < \frac{5}{4}.
\end{gather*}
As 
\eqref{eq:bilateral_obst_stat_sys_deg} 
and \cref{lem:dir_diff_uni} imply that 
$M^{\hat \phi}(v)$ and $M_{\hat \psi}(v)$
vanish a.e.\ in $Q \setminus V^\phi$
and $Q \setminus V_\psi$, respectively, 
for all $v \in L^1(L^1)$,
this shows that the functions 
$F$ and $G$ are indeed well defined as 
maps
$F\colon \Lambda \times B_r^U(\bar u)
\to \Lambda$
and
$G\colon \Lambda \times B_r^U(\bar u) \to Y$.
Next, we check the mapping properties 
of $M$, $F$, and $G$. 
That $M$ satisfies the Lipschitz 
condition \eqref{eq:LipM} follows immediately 
from \cref{lem:mult_contractive}. 
From \cref{lem:first_mod}, 
the fact that $M$ maps 
$B_r^U(\bar u)$ to $[L^\infty(L^2)]^2$,
and \cref{lem:mult_map_D}, we further immediately get that 
$M(u) = F(M(u),u)$ holds for all $u \in B_r^U(\bar u)$
(i.e., \eqref{eq.M_FP}).
Let us next verify the Lipschitz estimate for $F$ and $G$
in Assumption \ref{ass:standing:abstract}\ref{ass:standing:abstract:iv}b).
Suppose that 
$u \in B_r^U(\bar u)$ and 
$\lambda_\psi, \mu_\psi \in \MM(\overline{Q \cap V_\psi})$
are given and that 
$R_{m,n}\colon 
\MM(\overline{Q \cap V_\psi}) \to C_c^\infty( Q \cap V_\psi)$
is as in \cref{lemma:strict_density} for $V := Q \cap V_\psi$.
Then we get from \cref{lem:mult_map_D,lem_deg_obstacle_probs}
that 
\begin{equation*}
\begin{aligned}
& \|M^{\hat \phi}(u - D[\zeta_\psi D^{-1}R_{m,n}\lambda_\psi])
-
M^{\hat \phi}(u - D[\zeta_\psi D^{-1}R_{m,n}\mu_\psi])\|_{\MM(\overline{Q \cap V^\phi})}
\\
&=
 \|M^{\hat \phi}(u -  R_{m,n}\lambda_\psi)
-
M^{\hat \phi}(u -  R_{m,n}\mu_\psi) \|_{\MM(\overline{Q \cap V^\phi})}
\\
&=
 \|M^{\hat \phi}(u -  R_{m,n}\lambda_\psi)
-
M^{\hat \phi}(u -  R_{m,n}\mu_\psi) \|_{L^1(L^1)}
\\
&\leq
\| R_{m,n}\lambda_\psi 
- R_{m,n}\mu_\psi \|_{L^1(L^1)}
-
\| 
S^{\hat \phi}(u -  R_{m,n}\lambda_\psi)(T)
-
S^{\hat \phi}(u -  R_{m,n}\mu_\psi)(T)
\|_{L^1(\Omega)}
\\
&\leq
 \| R_{m,n}(\lambda_\psi 
- \mu_\psi)\|_{L^1(Q \cap V_\psi)}
-
\| 
S^{\hat \phi}(u -  R_{m,n}\lambda_\psi)(T)
-
S^{\hat \phi}(u - R_{m,n}\mu_\psi)(T)
\|_{L^1(\Omega^\phi)}
\\
&=
  \| R_{m,n}(\lambda_\psi 
- \mu_\psi)  \|_{\MM(\overline{Q \cap V_\psi})}
\\
&\qquad 
-
\| 
S^{\hat \phi}(u -  D[\zeta_\psi D^{-1}R_{m,n}\lambda_\psi])(T)
-
S^{\hat \phi}(u -  D[\zeta_\psi D^{-1}R_{m,n}\mu_\psi])(T)
 \|_{L^1(\Omega^\phi)}.
\end{aligned}
\end{equation*}
Due to the convergence properties 
in \cref{lemma:strict_density},
the continuity 
of the 
functions 
$M^{\hat \phi}\colon L^1(L^1) \to L^1(L^1)$
and 
$S^{\hat \phi}\colon L^1(L^1) \to C(L^1)$
from \cref{lem:dir_diff_uni},
the weak-star-to-strong continuity of maps 
of the type \eqref{eq:D_props},
and the definition of $D$,
we can let 
$m$ and then $n$ go to infinity in 
the above to obtain that 
\begin{align*}
&\|M^{\hat \phi}(u - D[\zeta_\psi D^{-1}\lambda_\psi])
-
M^{\hat \phi}(u - D[\zeta_\psi D^{-1}\mu_\psi])\|_{\MM(\overline{Q \cap V^\phi})}
\\
&\leq 
\| \lambda_\psi - \mu_\psi \|_{\MM(\overline{Q \cap V_\psi})}
\\
&\qquad - 
\| 
S^{\hat \phi}(u -  D[\zeta_\psi D^{-1} \lambda_\psi])(T)
-
S^{\hat \phi}(u -  D[\zeta_\psi D^{-1} \mu_\psi])(T)
 \|_{L^1(\Omega^\phi)}.
\end{align*}
Completely analogously, we also get that 
\begin{align*}
&\|M_{\hat \psi}(u - D[\zeta^\phi D^{-1}\lambda^\phi])
-
M_{\hat \psi}(u - D[\zeta^\phi D^{-1}\mu^\phi])\|_{\MM(\overline{Q \cap V_\psi})}
\\
&\leq 
 \| \lambda^\phi 
- \mu^\phi  \|_{\MM(\overline{Q \cap V^\phi})}
\\
&\qquad-
\| 
S_{\hat \psi}(u -  D[\zeta^\phi D^{-1}\lambda^\phi])(T)
-
S_{\hat \psi}(u -  D[\zeta^\phi D^{-1}\mu^\phi])(T)
\|_{L^1(\Omega_\psi)}
\end{align*}
holds for all $\lambda^\phi,  \mu^\phi \in \MM(\overline{Q \cap V^\phi})$.
Adding the last two estimates and 
comparing with the definitions of $F$ and $G$ 
shows that we have 
\begin{align*}
 \|
F(\lambda,u) - F(\mu,u)
 \|_\Lambda
+
  \|
G(\lambda,u) - G(\mu,u)
  \|_Y
\leq
 \|
\lambda  - \mu
 \|_\Lambda
\quad \forall \lambda, \mu \in \Lambda
\quad \forall u \in B_r^U(\bar u),
\end{align*}
as needed in Assumption \ref{ass:standing:abstract}\ref{ass:standing:abstract:iv}b). 
It remains to check the differentiability 
properties in Assumption \ref{ass:standing:abstract}\ref{ass:standing:abstract:iv}c).
To this end, 
we define $(\bar \lambda^\phi, \bar \lambda_\psi) := M(\bar u)$ and 
suppose that 
$g = (g^\phi, g_\psi) \in \Lambda$,
$h \in L^\infty(L^2)$, 
$\{g_n\} = \{(g_n^\phi, g_{n,\psi})\} \subset \Lambda$, and 
$\{\tau_n\} \subset (0, \infty)$
with $g_n \weaklystar g$ in $\Lambda$ and $\tau_n \to 0$
are given. 
Due to the definition of $F$, we have 
\begin{align*}
&\frac{F(M(\bar u) + \tau_n g_n, \bar u + \tau_n h) - F(M(\bar u),\bar u)}{\tau_n}
\\
&=
\frac{1}{\tau_n}
\left (
\begin{aligned}
M^{\hat \phi}(\bar u + \tau_n h- D[\zeta_\psi D^{-1}\bar \lambda_\psi] - \tau_n D[\zeta_\psi D^{-1}g_{n, \psi}])
&-
M^{\hat \phi}(\bar u - D[\zeta_\psi D^{-1}\bar \lambda_\psi])
\\
M_{\hat \psi}(\bar u + \tau_n h - D[\zeta^\phi D^{-1}\bar \lambda^\phi]
- \tau_n  D[\zeta^\phi D^{-1} g_{n}^\phi])
&-
M_{\hat \psi}(\bar u - D[\zeta^\phi D^{-1}\bar \lambda^\phi])
\end{aligned}
\right ).
\end{align*}
From the weak-star-to-strong continuity 
of functions of the type \eqref{eq:D_props}, 
we obtain that 
$D[\zeta_\psi D^{-1}g_{n, \psi}] \to D[\zeta_\psi D^{-1}g_\psi]$
and 
$D[\zeta^\phi D^{-1} g_{n}^\phi] \to D[\zeta^\phi D^{-1} g^\phi]$
in $L^1(L^1)$.
In view of the 
Hadamard directional differentiability 
properties of \smash{$M^{\hat \phi}$ and $M_{\hat \psi}$}
in  \eqref{eq:had_probs_unilatera},
the fact that 
\smash{$M^{\hat  \phi}$ and $M_{\hat \psi}$} vanish
a.e.\ in $Q \setminus V^\phi$, respectively, 
$Q \setminus V_\psi$, 
the inclusions 
$\overline{Q \cap V_\psi}, 
\overline{Q \cap V^\phi} \subset (0, T] \times \Omega$,
and the  Tietze extension theorem,
this yields
\begin{align*}
&\frac{F(M(\bar u) + \tau_n g_n, \bar u + \tau_n h) - F(M(\bar u),\bar u)}{\tau_n}
\\
&\weaklystar
\left (
\begin{aligned}
(M^{\hat \phi})'(\bar u - D[\zeta_\psi D^{-1}\bar \lambda_\psi]; 
h - D[\zeta_\psi D^{-1}g_\psi] 
)
\\
(M_{\hat \psi})'(\bar u - D[\zeta^\phi D^{-1}\bar \lambda^\phi];
h - D[\zeta^\phi D^{-1}g^\phi]
)
\end{aligned}
\right ) \quad \text{ in } \Lambda,
\end{align*}
as required in Assumption \ref{ass:standing:abstract}\ref{ass:standing:abstract:iv}c).
The differentiability assumption on $G$ is checked completely analogously,
using the properties of $S_{\hat \psi}$ and 
\smash{$S^{\hat \phi}$}
in \eqref{eq:had_probs_unilatera}. 
Having verified that
Assumption \ref{ass:standing:abstract} holds, we
now get from 
\cref{th:abstract_dir_diff}
that 
the map 
\[
S\colon B_r^U(\bar u) \to Y,
~
S(u):=
G(M(u), u)
=
\begin{pmatrix}
S^{\hat \phi}(u - D[\zeta_\psi D^{-1}\min(0, M_{\psi}^\phi(u))])(T)|_{\Omega^\phi}
\\
S_{\hat \psi}(u - D[\zeta^\phi D^{-1}\max(0, M_{\psi}^\phi(u))])(T)|_{\Omega_\psi}
\end{pmatrix},
\]
is directionally differentiable 
at $\bar u$ in all directions $h \in U$ in the sense of \eqref{eq:diff_claim}. 
Due to \cref{lem:first_mod,lem:mult_map_D},
we further know that
\begin{align*}
\begin{pmatrix}
S^{\hat \phi}(u - D[\zeta_\psi D^{-1}\min(0, M_{\psi}^\phi(u))])(T)|_{\Omega^\phi}
\\
S_{\hat \psi}(u - D[\zeta^\phi D^{-1}\max(0, M_{\psi}^\phi(u))])(T)|_{\Omega_\psi}
\end{pmatrix}
&=
\begin{pmatrix}
S^{\hat \phi}(u - \min(0, M_{\psi}^\phi(u)))(T)|_{\Omega^\phi}
\\
S_{\hat \psi}(u - \max(0, M_{\psi}^\phi(u)))(T)|_{\Omega_\psi}
\end{pmatrix}
\\
&=
\begin{pmatrix}
S_\psi^{\phi}(u)(T)|_{\Omega^\phi}
\\
S_{\psi}^\phi(u)(T)|_{\Omega_\psi}
\end{pmatrix}\quad \forall u \in B_r^U(\bar u).
\end{align*}
Since $\Omega_\psi \cup \Omega^\phi = \Omega$
holds
by our assumptions on $W_\psi$ and $W^\phi$, this shows that 
the function 
\smash{$L^\infty(L^2) \ni u \mapsto S_\psi^\phi(u)(T) \in L^1(\Omega)$} is directionally differentiable 
at $\bar u$ in all directions $h \in L^\infty(L^2)$.
This completes the proof. 
\end{proof}

Since we can always restrict the bilateral 
parabolic obstacle problem to 
a time interval $(0,s)$ with $0 < s < T$
and repeat the arguments of this section there, 
since $\bar u \in L^\infty(L^2)$ was an arbitrary given and fixed parameter, 
and since we have the Lipschitz continuity properties 
in \cref{lem:LipschitzPrimitiv,lem:Lipschitz},
we now obtain:

\begin{theorem}[differentiability for the bilateral 
parabolic obstacle problem]%
\label{th:main}%
Let $T$, $\Omega$, $Q$, $y_0$, $\psi$, and $\phi$
be as in Assumption \ref{ass:standing:obstacle}
and denote the solution map 
of the bilateral parabolic obstacle problem 
\eqref{eq:bilateral_obst_intro} by 
\begin{equation}
\label{eq:Sdef_again}
S_\psi^\phi\colon L^\infty(L^2) \to C(\bar Q) \cap L^2(H_0^1 \cap H^2)
\cap H^1(L^2),\qquad u \mapsto y.
\end{equation}
Then the function \smash{$S_\psi^\phi$} is Hadamard 
directionally differentiable in the sense that,
for all  $u, h \in L^\infty(L^2)$,
there exists $(S_\psi^\phi)'(u;h) \in L^\infty(Q) \cap L^2(H_0^1)$
such that 
\begin{equation}
\label{eq:main_claim}
\begin{aligned}
    &\{h_n\} \subset L^\infty(L^2), 
    \{\tau_n\} \subset \left (0,\infty \right ),
    h_n \to h \text{ in } L^\infty(L^2), 
    \tau_n \to 0
    \\
    &
    \Rightarrow
    ~
    \frac{S_\psi^\phi(u + \tau_n h_n) - S_\psi^\phi(u)}{\tau_n}
    \left \{
    \begin{aligned}
    &\to (S_\psi^\phi)'(u;h) \text{ in } L^p(Q) \text{ for all } 1 \leq p < \infty,
    \\
&\weaklystar (S_\psi^\phi)'(u;h) \text{ in } L^\infty(Q),
    \\
    &\weakly (S_\psi^\phi)'(u;h) \text{ in } L^2(H_0^1).
    \end{aligned}
    \right.
\end{aligned}
\end{equation}
\end{theorem}
\begin{proof}
From \cref{prop:dir_diff}, 
applied to restrictions of  
\eqref{eq:bilateral_obst_intro} to arbitrary 
intervals $(0,s)$, $0 < s \leq T$, we obtain that
the functions
\smash{$L^\infty(L^2) \ni u \mapsto S_\psi^\phi(u)(s) \in L^1(\Omega)$}
are directionally differentiable for all $s \in [0,T]$.
Due to the Lipschitz estimates in 
\cref{lem:LipschitzPrimitiv,lem:Lipschitz},
we further know that 
difference quotients of the function $S_\psi^\phi$
in \eqref{eq:Sdef_again} are bounded 
in $C(\bar Q) \cap L^2(H_0^1)$. 
Using the dominated convergence theorem 
and the theorem of Banach--Alaoglu, 
it now follows immediately that, for $h_n = h$,
the difference quotients in \eqref{eq:main_claim} 
converge as asserted. That one also has Hadamard directional 
differentiability, i.e., \eqref{eq:main_claim}
without the assumption $h_n = h$ for all $n \in \N$,
follows from the Lipschitz estimates
in \cref{lem:LipschitzPrimitiv,lem:Lipschitz} and 
\cite[Proposition~2.49]{BonnansShapiro2000}.
This completes the proof. 
\end{proof}

We remark that, in contrast to the elliptic case, 
it is not straightforward 
to derive an auxiliary EVI that uniquely characterizes 
the directional derivatives 
of the map $S_\psi^\phi$ in \eqref{eq:Sdef_again};
cf.\ \cite[Théorème 3.3]{Mignot1976},
\cite[Theorem 2]{Haraux1977},
and \cite[Theorem 4.1]{Christof2019}.
The results 
for vectorial sweeping processes obtained in 
\cite[Theorem 1.1]{BrokateChristof2026} suggest
that it might be possible to obtain a unique 
characterization of the functions 
$(S_\psi^\phi)'(u;h)$ by employing a 
Banach-space-valued version of 
Kurzweil--Stieltjes integration theory.
It is currently unknown, however, 
whether this approach indeed works
for \eqref{eq:bilateral_obst_intro}.

\bibliographystyle{siamplain}
\bibliography{references}

\end{document}